%% file: PF-arxiv_v2.tex
\documentclass[11pt]{article}
\usepackage{latexsym,amsthm,amsmath,amssymb,url,geometry}
\usepackage{thm-restate}
\usepackage{tikz}
\usetikzlibrary{decorations.pathreplacing}

\newcommand{\AY}[1]{{#1}}
\newcommand{\GG}[1]{{#1}}

\newcommand{\ModII}[2]{\mbox{$#1 \equiv #2$ (mod 2)}}
\newcommand{\2}{\vspace{0.15cm}}
\newtheorem{theorem}{Theorem}
\newtheorem{lemma}{Lemma}

\newtheorem{proposition}{Proposition}
\newtheorem{definition}{Definition}

\newcommand{\SpL}{\hspace{0.1cm}}

\begin{document}

\title{Perfect Forests in Graphs and Their Extensions} 

\author{Gregory Gutin$^{1}${ }
Anders Yeo$^{2,3}$
\\$^{1}$ Department of Computer Science\\ Royal Holloway, University of London\\ Egham, United Kingdom g.gutin@rhul.ac.uk \\
$^{2}$ Department of Mathematics and Computer Science\\ University of Southern Denmark, Denmark {andersyeo@gmail.com} 
\\$^{3}$ Department of Pure and Applied Mathematics\\
University of Johannesburg, South Africa}
%


\maketitle

\begin{abstract}
Let $G$ be a graph on $n$ vertices. For $i\in \{0,1\}$ and a connected graph $G$, a spanning forest $F$ of $G$ is called
an  $i$-perfect forest if every tree in $F$ is an induced subgraph of $G$ and exactly $i$ vertices of $F$ have even degree (including zero). 
A  $i$-perfect forest of $G$ is proper if it has no vertices of degree zero. Scott (2001) showed that
every connected graph with even number of vertices contains a (proper) 0-perfect forest. 
We prove that one can find a 0-perfect forest with minimum number of edges in polynomial time, but it is NP-hard to obtain a 0-perfect forest with maximum number of edges.
\GG{Moreover, we show that to decide whether $G$ has a 0-perfect forest with at least $|V(G)|/2+k$ edges, where $k$ is the parameter, is W[1]-hard. }
We also prove that for a prescribed edge $e$ of $G,$ it is NP-hard to obtain a 0-perfect forest containing $e,$ but one can \AY{decide if there exists
 a 0-perfect forest not containing $e$ in polynomial time.}
It is easy to see that every graph with odd number of vertices has a 1-perfect forest. It is not the case for proper 1-perfect forests.
We give a characterization of when a connected graph has a proper 1-perfect forest.
\end{abstract}

\newpage

\section{Introduction}

In this paper all graphs are finite, undirected, have no parallel edges or loops. We use standard terminology and notation, see e.g. \cite{Diestel}.
The number of vertices (edges, respectively) of a graph $G$ is called its {\em order} ({\em size}, respectively). The degree of a vertex $x$ in a graph $G$ is denoted by $d_G(x).$
A vertex $x$ of a graph $G$ is {a} {\em cut-vertex}  if $G-x$ has more connected components than $G.$
A maximal connected subgraph of a graph $G$ without a cut-vertex is called a {\em block}. 
Thus, every block of $G$ is either a maximal 2-connected subgraph or a bridge (including its vertices) or an isolated vertex, implying that
a block of odd order in a connected graph of order at least 3, must be a maximal 2-connected subgraph.

A spanning forest $F$ of $G$ is called a {\em semiperfect forest} if every tree of $F$ is an induced subgraph of $G.$
Let $G$ be a graph and let $f\colon V(G) \rightarrow \{0,1\}$ be a function such that $\sum_{v \in V(G)} f(v)$ is even (we will call such a function {\em even-sum}). 
A subgraph $H$ in $G$ where  \ModII{d_H(x)}{f(x)} for all $x \in V(G),$ is called an $f$-{\em parity subgraph}. Note that the requirement that $f$ is even-sum is 
necessary as otherwise an $f$-parity subgraph does not exist. 
An $f$-parity subgraph $H$ of $G$ is called an {\em $f$-parity perfect forest} if $H$ is a semiperfect forest. 

For $i\in \{0,1\}$ and a graph $G$, an $f$-parity perfect forest is called an {\em $i$-perfect forest} if $f(x)=1$ for all vertices of $G$ for $i=0$, and for all vertices of $G$ apart from one for $i=1.$
An  $i$-perfect forest of $G$ is {\em proper} if it has no vertices of degree zero. 
Note that every 0-perfect forest (called a perfect forest in \cite{CaroLZ,Gut} and a pseudo-matching in \cite{SW})
is proper. {For examples of 0-perfect and 1-perfect forests, see Figures \ref{0perfect} and \ref{1perfect}.}

\begin{figure}[hbtp]
\begin{center}
\tikzstyle{vertexL}=[circle,draw, top color=gray!5, bottom color=gray!30,minimum size=12pt, scale=0.6, inner sep=0.1pt]
\begin{tikzpicture}[scale=0.45]
\draw (3,-0.5) node {{\footnotesize (a): $G$}};
\node (x1) at (1,3) [vertexL] {};
\node (x2) at (3,3) [vertexL] {};
\node (x3) at (5,3) [vertexL] {};
\node (x4) at (1,1) [vertexL] {};
\node (x5) at (3,1) [vertexL] {};
\node (x6) at (5,1) [vertexL] {};
\draw [line width=0.02cm] (x1) to (x2);
\draw [line width=0.02cm] (x1) to (x5);
\draw [line width=0.02cm] (x2) to (x3);
\draw [line width=0.02cm] (x2) to (x4);
\draw [line width=0.02cm] (x3) to (x5);
\draw [line width=0.02cm] (x4) to (x5);
\draw [line width=0.02cm] (x5) to (x6);
\end{tikzpicture} \hspace{1cm}
\begin{tikzpicture}[scale=0.45]
\draw (3,-0.5) node {{\footnotesize (b): A $0$-perfect forest of $G$}};
\node (x1) at (1,3) [vertexL] {};
\node (x2) at (3,3) [vertexL] {};
\node (x3) at (5,3) [vertexL] {};
\node (x4) at (1,1) [vertexL] {};
\node (x5) at (3,1) [vertexL] {};
\node (x6) at (5,1) [vertexL] {};
\draw [line width=0.02cm] (x1) to (x2);
\draw [line width=0.02cm] (x2) to (x3);
\draw [line width=0.02cm] (x2) to (x4);
\draw [line width=0.02cm] (x5) to (x6);
\end{tikzpicture}
\end{center}

\caption{A graph $G$ is shown in (a) and a $0$-perfect forest of $G$ is shown in (b) 
(as all degrees are odd and the trees are induced in $G$).} \label{0perfect}
\end{figure}
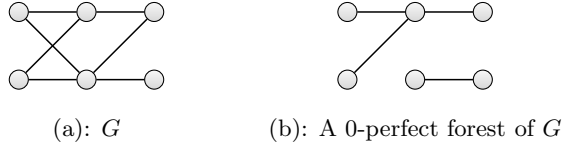

\begin{figure}[hbtp]
\begin{center}
\tikzstyle{vertexL}=[circle,draw, top color=gray!5, bottom color=gray!30,minimum size=12pt, scale=0.6, inner sep=0.1pt]
\begin{tikzpicture}[scale=0.45]
\draw (3,-0.5) node {{\footnotesize (a): $H$}};
\node (x1) at (1,3) [vertexL] {};
\node (x2) at (3,3) [vertexL] {};
\node (x3) at (5,3) [vertexL] {};
\node (x4) at (1,1) [vertexL] {};
\node (x5) at (3,1) [vertexL] {};
\node (x6) at (5,1) [vertexL] {};
\node (x7) at (6.5,2) [vertexL] {};
\draw [line width=0.02cm] (x1) to [out=30, in=150] (x3);
\draw [line width=0.02cm] (x1) to (x2);
\draw [line width=0.02cm] (x1) to (x4);
\draw [line width=0.02cm] (x2) to (x4);
\draw [line width=0.02cm] (x2) to (x5);
\draw [line width=0.02cm] (x2) to (x6);
\draw [line width=0.02cm] (x3) to (x6);
\draw [line width=0.02cm] (x3) to (x7);
\draw [line width=0.02cm] (x6) to (x7);
\end{tikzpicture} \hspace{0.5cm}
\begin{tikzpicture}[scale=0.45]
\draw (3,-0.5) node {{\footnotesize (b): A $1$-perfect forest of $H$}};
\node (x1) at (1,3) [vertexL] {};
\node (x2) at (3,3) [vertexL] {};
\node (x3) at (5,3) [vertexL] {};
\node (x4) at (1,1) [vertexL] {};
\node (x5) at (3,1) [vertexL] {};
\node (x6) at (5,1) [vertexL] {};
\node (x7) at (6.5,2) [vertexL] {};
\draw [line width=0.02cm] (x2) to (x4);
\draw [line width=0.02cm] (x2) to (x5);
\draw [line width=0.02cm] (x2) to (x6);
\draw [line width=0.02cm] (x3) to (x7);
\end{tikzpicture} \hspace{0.2cm}
\begin{tikzpicture}[scale=0.45]
\draw (3,-0.5) node {{\footnotesize (c): A proper $1$-perfect forest of $H$}};
\node (x1) at (1,3) [vertexL] {};
\node (x2) at (3,3) [vertexL] {};
\node (x3) at (5,3) [vertexL] {};
\node (x4) at (1,1) [vertexL] {};
\node (x5) at (3,1) [vertexL] {};
\node (x6) at (5,1) [vertexL] {};
\node (x7) at (6.5,2) [vertexL] {};
\draw [line width=0.02cm] (x1) to [out=30, in=150] (x3);
\draw [line width=0.02cm] (x2) to (x4);
\draw [line width=0.02cm] (x2) to (x5);
\draw [line width=0.02cm] (x2) to (x6);
\draw [line width=0.02cm] (x3) to (x7);
\end{tikzpicture} 
\end{center}
\caption{The graph $H$ is shown in (a), a (non-proper) $1$-perfect forest of $H$ is shown in (b), and a proper 
$1$-perfect forest of $H$ is shown in (c).} \label{1perfect}
\end{figure}
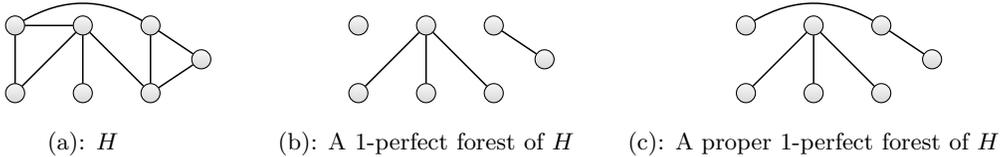

Clearly, every connected graph with a 0-perfect forest is of even order. Scott \cite{Sco2001} proved that somewhat surprisingly the opposite implication is also true.
\begin{theorem} \label{ThmKnown}
Every connected graph of even order contains a 0-perfect forest. 
\end{theorem}
The proof of Theorem \ref{ThmKnown} in  \cite{Sco2001} is  graph-theoretical and relatively long. A short proof using basic linear algebra is obtained in \cite{Gut} and two short graph-theoretical proofs are given in \cite{CaroLZ}. 
All the proofs of Theorem \ref{ThmKnown} are constructive and yield polynomial algorithms for finding 0-perfect forests.
Intuitively, it is clear that a 0-perfect forest can provide a useful structure in a graph and, in particular, this notion was used by Sharan and Wigderson \cite{SW} to prove that the perfect matching problem for bipartite cubic graphs belongs to the complexity class ${\cal NC}$. {Semiperfect forests were used in the proofs of three theorems in \cite{GY2021lower}.}
Gutin and Yeo \cite{GutYeo} studied extensions of a 0-perfect forest to directed graphs.

Since a 0-perfect forest is a generalization of a {matching, it} is natural to study the following two problems for a connected graph $G$ of even order $n$: 

(1) Find a 0-perfect forest of $G$ of minimum size. (Clearly, the minimum size is $n/2$ if and only if $G$ has a perfect matching.)

(2) Find a 0-perfect forest of $G$ of maximum size. (This is of interest in matching-like edge-decompositions of $G.$)

\2

The following theorem  solves the first problem. 

\begin{theorem} \label{CorI}
In  polynomial time, we can find a $0$-perfect forest of minimum size.
\end{theorem}

Theorem \ref{CorI} follows immediately from the next theorem by letting $f(x)=1$ for all $x \in V(G)$. Theorem  \ref{ThmII} shows usefulness of extending
Problem 1 to $f$-parity perfect forests. Theorem  \ref{ThmII} is proved in Section \ref{sec:min}.

\begin{restatable}{theorem}{thmm} 
\label{ThmII}
Let $G$ be a connected graph and let $f\colon V(G) \rightarrow \{0,1\}$ be an even-sum function.
We can in polynomial time find an $f$-parity perfect forest $H$ in $G$,  such that  \ModII{d_H(x)}{f(x)} for all $x \in V(G)$ 
{and $|E(H)|$ is minimized.}
\end{restatable}


As the following theorem  shows, the second problem cannot be solved in polynomial time unless {\sf P}={\sf NP}.  

\begin{theorem}\label{thm:NP} 
It is {\sf NP}-hard to find a $0$-perfect forest of maximum size. 
\end{theorem}

Let $n=|V(G)|.$ Theorem \ref{thm:NP}  follows from the next result proved in Section \ref{sec:max}. Theorem \ref{ThmI} is optimal in the following sense. 
The problem of finding a $0$-perfect forest of size at least $n-1$ is polynomial-time solvable because $G$ has a $0$-perfect forest of size at least $n-1$ 
if and only if $G$ is a tree in which every vertex is of odd degree. 

\begin{theorem} \label{ThmI}
It is {\sf NP}-hard to decide whether a connected graph contains a $0$-perfect forest with at least $n-2$ edges.
\end{theorem}

It is easy to show that Theorem \ref{ThmI} holds if we replace $n-2$ by $n-k$ for any integer $k\ge 2.$
Indeed, add two new vertices $x$ and $y$ to a graph $G$ as well as two edges $xy$ and $yu,$ where $u$ is any vertex in $G.$ The resulting graph is denoted by $G'.$
Observe that there is a 0-perfect forest of size $|V(G)|-k$  in $G$ if and only if there is a 0-perfect forest of size $|V(G')|-(k+1)$ in $G'.$ 

Since the problem of finding a $0$-perfect forest of maximum size is {\sf NP}-hard, 
\GG{
it is {natural} to study parameterized complexity of the problem; we will provide a short introduction to parameterized algorithms and complexity in Section \ref{sec:W1}, for excellent introductions
to the area, see e.g. \cite{Cygan+,downey2013,FlumG06}. Since $n/2$ is a tight lower bound and $n-1$ is a tight upper bound for the maximum size, it is natural to consider
  below-tight-upper-bound and above-tight-lower-bound parameterizations of the problem.\footnote{Such parameterizations were studied for many graph-theoretical and constraint satisfaction problems, see e.g. \cite{Bang+,CJM,GY,LokshtanovNRRS14,LSSZ}.} 
 In other words, we can ask whether there is a $0$-perfect forest of size at least $n-k$ ($n/2+k$, respectively), where $k$ is the parameter.
Theorem \ref{ThmI} shows that the parameterization {$n-k$  is \sf para-NP}-complete. 
In the conference version \cite{GY2021perfect} of this paper, we asked whether the parameterization {$n/2+k$} is fixed-parameter tractable. 
The following theorem proved in Section \ref{sec:W1} implies 
that this is highly unlikely as there are a number of reasons to believe that no  {\sf W}[1]-hard parameterized problem is fixed-parameter tractable (see e.g. the discussion on the Exponential Time Hypothesis in Section \ref{sec:W1}). Henceforth we will call the problem of deciding whether a connected graph $G$ contains a $0$-perfect forest with at least $n/2+k$ edges the {\sc Perfect Forest Above Perfect Matching} problem.  

\begin{theorem} \label{ThmW}
{\sc Perfect Forest Above Perfect Matching}  is {\sf W}[1]-hard.
\end{theorem}
}


Here is another pair of natural problems on 0-perfect forests. They both are clearly polynomial-time solvable when restricted to perfect matchings. For a graph $G$ of even order and an edge $e$ in $G$,

($1'$) find a 0-perfect forest containing $e$; 

($2'$) find a 0-perfect forest not containing $e.$

\2

For Problem $1',$ we prove the following result in Section \ref{sec:ince}.

\begin{restatable}{theorem}{thsix} 
\label{ZeroPerfectForestContainingEdge}
The following problem is {\sf NP}-hard. Given a connected graph $G$ and an edge $e \in E(G),$ decide whether 
 $G$ has a $0$-perfect forest containing $e$.
\end{restatable}

For Problem $2',$ we have the next result, which  follows immediately from Theorem~\ref{FparityForestAvoidingEdge}, by letting $f(x)=1$ for
all $x$ in $G$. 
Theorem~\ref{FparityForestAvoidingEdge} again demonstrates usefulness of $f$-parity perfect forests. It is proved in Section \ref{sec:noe}.

\begin{theorem} \label{ZeroPerfectForestAvoidingEdge}
Given a graph $G$ and an edge $e \in E(G)$
we can in polynomial time decide whether $G$ has a $0$-perfect forest not containing $e$.
\end{theorem}

\begin{theorem} \label{FparityForestAvoidingEdge}
The following problem is polynomial time solvable. Given a graph $G$, an edge $e \in E(G)$ and an even-sum 
function $f\colon V(G) \rightarrow \{0,1\}$, decide whether
 $G$ has an $f$-parity perfect forest not containing $e$.
\end{theorem}

\2

Since an odd order connected graph cannot have a 0-perfect forest, 
it is natural to ask whether every connected graph of odd order has a 1-perfect forest (recall that a 1-perfect forest has only one vertex of even degree).
The answer is positive and the proof is trivial. In fact, it is not hard to show the following strengthening of this observation, which will be useful in the proof of Theorem \ref{main}.  

\begin{proposition}\label{Thm1perfect}
Let $x$ be an arbitrary vertex of a connected graph $G$ of odd order. Then $G$ has a 1-perfect forest $F$ such that $d_F(x)$ is even. 
\end{proposition} 
\begin{proof}
Create a new graph $H$ by adding a new vertex $y$ to $G$ and adding the edge $xy$.
By Theorem~\ref{ThmKnown},  $H$ has a 0-perfect forest, $F_H$. 
Deleting the vertex $y$ from $F_H$, results in  
the desired 1-perfect forest of $G$ where $x$ is the only vertex of even degree.
\end{proof}

Note that not every connected graph of odd order has a proper 1-perfect forest. For example, no complete graph of odd order has such a forest. Thus, a more interesting question with a potentially more useful answer is when a connected graph of odd order has a proper 1-perfect forest? This question is answered in the following characterization proved in Section \ref{sec:main}.

\begin{theorem}\label{main}
Let ${\cal B}$ be the set of all connected graphs where every block is a complete graph of odd order.
If $G$ is a connected graph of odd order $n\ge 3$ then $G$ contains a proper $1$-perfect forest  if and only if 
$G \not\in {\cal B}$.
\end{theorem}

Using this theorem and a linear-time algorithm for computing biconnected components in a graph \cite{HopTar}, in polynomial time we can decide whether a connected graph $G$  of odd order contains a proper $1$-perfect forest. If $G\not\in {\cal B}$, the proof by induction of Theorem \ref{ThmMain} yields a polynomial-time recursive algorithm to construct a proper 1-perfect forest. 

Our proof of Theorem \ref{main} is graph-theoretical and so are the proofs of Theorem \ref{ThmKnown} in \cite{Sco2001} and  \cite{CaroLZ}.
Recall that Gutin \cite{Gut} gave a linear-algebraic proof of Theorem \ref{ThmKnown}. It would interesting to see whether Theorem \ref{main} can {be} proved  using a linear-algebraic approach, too.

\GG{This paper is an extended version of conference paper \cite{GY2021perfect}. Here we solve an open problem from \cite{GY2021perfect} (see Theorem \ref{ThmW}) and provide a proof of Theorem \ref{inducedCycle} omitted in \cite{GY2021perfect}.}

\section{Proof of Theorem  \ref{ThmII}}\label{sec:min}

\begin{lemma} \label{LemI}
Let $G$ be a connected graph and let $f\colon V(G) \rightarrow \{0,1\}$ be an even-sum function.
If $H$ is an $f$-parity subgraph of $G$ of minimum size, then $H$ is an $f$-parity perfect forest.
\end{lemma}
\begin{proof}
Assume that $H$ is an $f$-parity subgraph with minimum possible $|E(H)|$.
Clearly $H$ contains no cycles, as removing the edges of a cycle would contradict the minimality of $|E(H)|$.
Assume that some tree $T$ of $H$ is not an induced tree in $G$.
Let $xy$ be an edge of $G$, not belonging to $T$ but with $\{x,y\} \subseteq V(T)$.
Remove the unique $(x,y)$-path in $T$ from $H$ and add the edge $xy$ to $H$.
This decreases the number of edges in $H$ without changing the parity of the degree of any vertex, 
contradicting the minimality of $|E(H)|$. Therefore $H$ is indeed an $f$-parity perfect forest.
\end{proof}

Lemma~\ref{LemI}  implies the following:

\begin{theorem}\label{th:fpar}
Let $G$ be a connected graph and let $f\colon V(G) \rightarrow \{0,1\}$ be an even-sum function.
Then there exists an $f$-parity perfect forest $F$ in $G$. 
\end{theorem}
\begin{proof} 
Let $x_1,x_2,\ldots,x_{k},y_1,y_2,\ldots,y_{k}$ be the vertices in $G$ with $f$-value equal to one.
Let $P_i$ be any $(x_i,y_i)$-path in $G$ for all $i=1,2,\ldots,k$, which exists as $G$ is connected.
Let $H$ be the spanning subgraph of $G$ such that an edge $e \in E(G)$ belongs to $H$ if and only if
$e$ belongs to an odd number of paths in $P_1,P_2,\ldots,P_{k}$. Let $x\in V(G).$ Observe that
$d_H(x)$ is odd if and only if $x$ is incident with an odd number of edges in $\cup_{i=1}^{k} E(P_i)$, which is
if and only if $x$ is the endpoint of one of the paths i.e. $f(x)=1.$ Thus, $H$ is an $f$-parity subgraph of $G.$
Lemma~\ref{LemI} now implies that if $H$ is the $f$-parity subgraph
of $G$ of minimum size, then $H$ is an $f$-parity perfect forest.
\end{proof}

Note that Theorem \ref{th:fpar} generalizes Theorem~\ref{ThmKnown}:
set $f(x)=1$ for all $x \in V(G).$ Thus, Theorem \ref{th:fpar} provides an alternative proof of Theorem~\ref{ThmKnown}.

\2


\thmm*
\begin{proof}
Let $G$ be a connected graph and let $f\colon V(G) \rightarrow \{0,1\}$ be an even-sum function.
Let $V(G)=\{v_1,v_2,\ldots,v_n\}$.
We will construct a weigthed auxillary graph $H$ as follows.  Let $V(H) = \cup_{i=1}^{n} X_i,$ where  for every $i\in [n]$, $|X_i| \in \{n-1,n\}$
and \ModII{|X_i|}{f(v_i)}.
For all $1 \leq i < j \leq n$ and all $u \in X_i$ and $v \in X_j,$ we let $uv \in E(H)$ if and only if $v_iv_j \in E(G)$.
Finally add a matching $M_i = \{e_1^i,e_2^i,\ldots,e_{\lfloor |X_i|/2 \rfloor}^i\}$  to $X_i$ for all $i \in [n]$.  Let the weight of all the edges within each $X_i$ (i.e. the edges in $M_i$) be zero and let all edges between 
different $X_i$'s have weight one.

  We first show that $H$ contains a perfect matching. As $\sum_{v \in V(G)} f(v)$ is even we may assume 
that $\{v_1,v_2,\ldots,v_{2k}\}$ are the vertices of $G$ with an $f$-value of one for some integer $k$ with $0 \leq k \leq n/2$.
Assume that $y_i \in X_i$ is the unique vertex in $X_i$ that is not saturated by $M_i$ for all $i \in [2k]$
and start of by letting $M$ be the matching containing all $M_i$'s.

Let $P_i = v_i v_{p_1^i} v_{p_2^i} \cdots v_{p_{l_i-1}^i} v_{i+k}$ be any path in $G$ from $v_i$ to $v_{i+k}$ where $i \in [k]$. 
It is not difficult to see that there exists an $M$-augmenting path, $Q_i$, in $H$ starting in $y_i$ 
and ending in $y_{i+k}$ and containing exactly the edges $e_i^{p_1^i}$, $e_i^{p_2^i}$, ..., $e_i^{p_{l_i-1}^i}$
from $M$. Also observe that $Q_1,Q_2,\ldots, Q_k$ are vertex disjoint, which implies that we can use all $Q_i$ to
increase the matching $M$ thereby obtaining a perfect matching in $H$.

We will now show the following claim. The {\em size} of a multiset $S$ is the total number of elements in $S,$ where if an element $e\in S$ 
is of multiplicity $r,$ then $e$ is counted $r$ times.

\2


 {\bf Claim~A:} {(a)}{\em If there exists a perfect matching in $H$ with weight $w^*$ then there exists a
 multiset of edges $E^*$ in $G$ of size $w^*$, such that $\ModII{d_{E^*}(x)}{f(x)}$ for all $x \in V(G)$.

{(b)} Conversely if  $E^*$ is a multiset of edges in $G$ of size $w^*$, such that  $\ModII{d_{E^*}(x)}{f(x)}$ for all $x \in V(G)$, then there exists a 
 perfect matching in $H$ with weight at most $w^*$.}

\2

{\em Proof of Claim~A:} First assume that we have a multiset of edges $E^*$ in $G$ of size $w^*\le W_{\max}$, such that 
\ModII{d_{E^*}(x)}{f(x)} for all $x \in V(G)$.
Let $M^*=\emptyset$. For every $v_i v_j \in E^*$ we will add edges between $X_i$ and $X_j$ to $M^*$ as follows: if $v_i v_j$ 
is of multiplicity {$r$ in $E^*,$  then we add an edge between $X_i$ and $X_j$ to $M^*$ if and only if $r$ is odd.}
Since we will add $2k_i+f(v_i)$ edges that are incident to $X_i$ for each $i\in [n]$ (where $k_i$ is some integer), we can 
add these edges such that their endvertices  are $V(e_1^i) \cup V(e_2^i) \cup \cdots \cup V(e_{k_i}^i)$ if $f(v_i)=0$ 
and $\{y_i\} \cup V(e_1^i) \cup V(e_2^i) \cup \cdots \cup V(e_{k_i}^i)$ if $f(v_i)=1$ for each $i\in [n],$ where $V(e^i_j)$ denotes the pair of endvertices of $e^i_j.$
We can now extend $M^*$ to a perfect matching by adding $M_i \setminus \{e_1^i,e_2^i,\ldots,e_{k_i}^i\}$ for each $i \in [n]$.
This gives us a perfect matching in $H$ with weight at most $|E^*|$ as desired.

Conversely assume that there exists a perfect matching $M^*$ in $H$ with weight $w^*$.  
Initially let $E^* = \emptyset$. For every $xy \in M^*$ with weight one (i.e. $x\in X_i$ and $y\in X_j$ for some $i\ne j$), 
add $v_iv_j$ to $E^*$. This gives us the desired multiset $E^*$, thereby completing the proof 
of Claim~A.

\2

We have proved that $H$ has a perfect matching. Let $M_{\min}$ be a minimum weight perfect matching in $H$ which can be determined
in polynomial time using Edmonds' blossom algorithm as a subroutine, see e.g. \cite{LovPla}. 
{
Let $W_{\min}$ be the weight of $M_{\min}.$
By Claim~A(a), using $M_{\min},$ in polynomial time we can find a multiset of edges $E^*$ in $G$ of size $W_{\min}$ , such that \ModII{d_{E^*}(x)}{f(x)} for all $x \in V(G)$.
By Claim~A(b), since $W_{\min}$ is the minimum weight of a perfect matching in $H$, $W_{\min}$ is  minimum size of a multiset of edges $E^{**}$, such that \ModII{d_{E^{**}}(x)}{f(x)} for all $x \in V(G)$.
}

Note that no edge is in $E^*$ more than once, 
since if some edge, $e$, appears twice, then we can delete two copies of $e$ from $E^*$, thereby contradicting the 
minimality of $|E^*|$. Let $F$ be the spanning subgraph of $G$ with edge set $E^*$.
By Lemma~\ref{LemI} we note that $F$ is an $f$-parity perfect forest, which completes the proof of the theorem.
\end{proof}

\section{Proof of Theorem \ref{ThmI}}\label{sec:max}

We will reduce from the {\em not-all-equal 3-SAT} problem, abbreviated to NAE-3-SAT, which is the problem of determining whether an instance of $3$-SAT has a truth assignment
to its variables such that every clause contains both a true and a false literal. If this is the case we say that the instance is {\em NAE-satisfied}.
 NAE-3-SAT is known to be NP-hard to solve \cite{Sch}.
Let $I$ be an instance of NAE-3-SAT with clauses $C_1,C_2,\ldots,C_m$ and variables $v_1,v_2,\ldots,v_n$. We will construct a graph $G$ such that
$G$ contains a $0$-perfect forest with at least $n-2$ edges if and only if $I$ is NAE-satisfied.

We first create a gadget $H_i$ for each $i=1,2,\ldots,n$ as follows. Let $$V(H_i) = \{x_1^i,z_1^i,y_1^i,x_2^i,z_2^i,y_2^i\}$$ and add all possible edges to $H_i$,
except $x_1^iy_1^i$ and $x_2^iy_2^i$. For all $i=1,2,\ldots, n-1$ we then add all edges between 
$\{y_1^i,y_2^i\}$ and $\{x_1^{i+1}, x_2^{i+1}\}$.
Now add a pendent edge to each vertex in $V(H_i) \setminus \{x_1^1,x_2^1,y_1^n,y_2^n\}$ for all $i=1,2,\ldots,n$.
 See Figure~\ref{gadgetHs} for an illustration of this part of $G$, which is denoted by $Q$.
\begin{figure}[hbtp]
\input{pic1.tex}
\caption{The gadgets $H_1,H_2, \ldots, H_n$ and the edges connecting these. The resulting graph is denoted by $Q$.} \label{gadgetHs}
\end{figure}
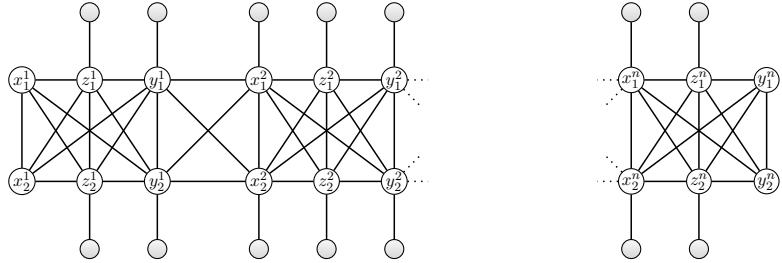
We will now complete our construction of $G$. \newline Let 
$V(G) = V(Q) \cup \{c_1,c_2,\ldots,c_m\} \cup \{c_1',c_2',\ldots,c_m'\}.$
For each $j=1,2,\ldots,m$ we will add an edge from both $c_j$ and $c_j'$ to $y_2^i$ if and only if $v_i$ is 
a literal in the clause $C_j$.  We will furthermore add an edge from both $c_j$ and $c_j'$ to $y_1^i$ if and only $\overline{v_i}$ is 
a literal in the clause $C_j$. This completes the construction of $G$. See Figure~\ref{example} depicting $G$ for $I=(v_1, v_2,\overline{v_3})$.
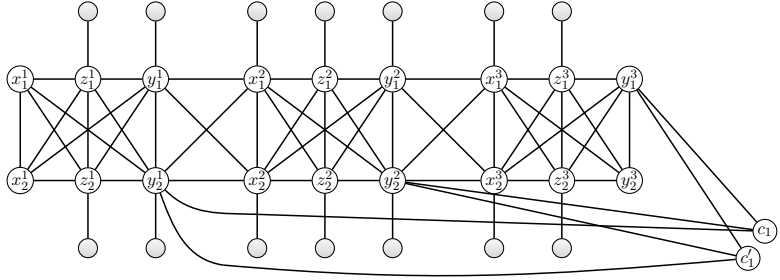
\begin{figure}[hbtp]
\input{pic2.tex}
\caption{The graph $G$ if $I=(v_1,v_2,\overline{v_3})$.} \label{example}
\end{figure}

We will now show that $G$ contains a $0$-perfect forest of size at least $n-2$ if and only if $I$ is NAE-satisfied.
First assume that $I$ is NAE-satisfied and consider a truth assignment $\tau$ NAE-satisfying $I.$ We will construct two vertex-disjoint induced trees, 
$T_1$ and $T_2$, in $G$, such that all degrees in the trees $T_i$ are odd for $i \in [2]$. If $v_i$ is true in $\tau$ then add the vertices in $\{x_1^i,z_1^i,y_1^i\}$ to 
$T_1$ and the vertices in $\{x_2^i,z_2^i,y_2^i\}$ to  $T_2$. Conversely, if $v_i$ is false in $\tau$ then add the vertices in $\{x_1^i,z_1^i,y_1^i\}$ to
$T_2$ and the vertices in $\{x_2^i,z_2^i,y_2^i\}$ to  $T_1$. We furthermore add all vertices of degree one to the same tree as their
neighbour. Note that the vertices we have added so far to $T_i$ (for $i \in [2]$) induce a tree in $G,$ where every vertex has odd degree in $T_i$.

Finally as $I$ is NAE-satisfied we note for $j \in [m]$, each of $c_j$ and $c_j'$ has one edge into one of the $T_i$'s and
two edges into the other $T_i$. Add each of $c_j$ and $c_j'$ to the $T_i$ with which it is only connected by one edge. 
We note that after this operation the vertices we have added so far to $T_i$ (for $i \in [2]$) still induces a tree in $G$ 
where every vertex has odd degree in $T_i$. After doing the above operation for all  $j \in [m]$ we have obtained the desired 
trees $T_1$ and $T_2$ whose union form a $0$-perfect forest in $G$ with $|V(G)|-2$ edges. See Figure~\ref{exampleSOL} for the found
$T_1$ and $T_2$ if the instance of NAE-3-SAT is $I=(v_1,v_2,\overline{v_3})$ and the truth assignment is to set all variables equal to true.
\begin{figure}[hbtp]
\input{pic3.tex}
\caption{The trees $T_1$ and $T_2$ if $I=(v_1,v_2,\overline{v_3})$ and $v_1=v_2=v_3=true$.} \label{exampleSOL}
\end{figure}
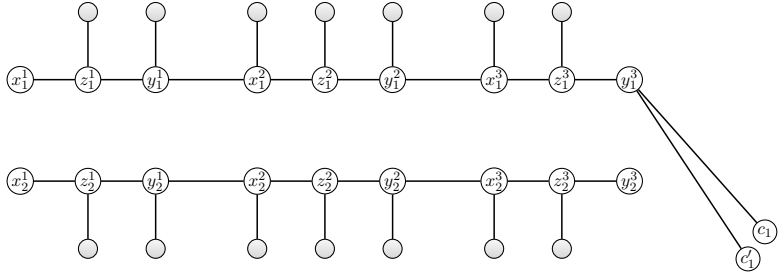

Conversely, assume that $G$ contains a $0$-perfect forest with at least $|V(G)|-2$ edges. As $G$ is not a tree 
this implies that $G$ contain two vertex-disjoint trees $T_1$ and $T_2$ such that each $T_i$ is an induced tree in $G$ of order at least 2,
all degrees in each $T_i$ are odd, and $V(T_1)$ and $V(T_2)$ partition $V(G)$. We will now prove the following claims where Claim~C completes 
the proof of the theorem.

\2

{\bf Claim~A:} {\em For each $i \in [n]$ one of the following cases hold.
\begin{description}
 \item[A.1:] $\{x_1^i,z_1^i,y_1^i\} \in V(T_1)$ and $\{x_2^i,z_2^i,y_2^i\} \in V(T_2)$.
 \item[A.2:] $\{x_1^i,z_2^i,y_1^i\} \in V(T_1)$ and $\{x_2^i,z_1^i,y_2^i\} \in V(T_2)$.
 \item[A.3:] $\{x_1^i,z_1^i,y_1^i\} \in V(T_2)$ and $\{x_2^i,z_2^i,y_2^i\} \in V(T_1)$.
 \item[A.4:] $\{x_1^i,z_2^i,y_1^i\} \in V(T_2)$ and $\{x_2^i,z_1^i,y_2^i\} \in V(T_1)$.
\end{description}
}

\2

{\em Proof of Claim~A:} 
As the only two non-edges in $H_i$ are $x_1^i y_1^i$ and $x_2^i y_2^i$ we note that there 
exist a $4$-cycle on every set of $4$ vertices in $H_i$. Therefore $|V(T_j) \cap V(H_i)| \geq 4$ is
not possible for any $j \in [2]$ and $i \in [n]$. So $|V(T_j) \cap V(H_i)|=3$ for $j \in [2]$ and $i \in [n]$.

As there is no $3$-cycle in $G[V(T_j)]$ for $j \in [2]$ we note that $x_1^i$ and $y_1^i$ must belong to one of 
the trees, say $T_j$, and $x_2^i$ and $y_2^i$ must belong to the other tree, $T_{3-j}$. 
So if $x_1^i \in V(T_1)$ then $y_1^i \in V(T_1)$ and $\{x_2^i,y_2^i\} \subseteq V(T_2)$ and we are in case A.1 or A.2.
On the other hand if  $x_1^i \in V(T_2)$ then $y_1^i \in V(T_2)$ and $\{x_2^i,y_2^i\} \subseteq V(T_1)$ and we are in case A.3 or A.4.
This completes the proof of Claim~A.

\2

{\bf Claim~B:} {\em For $i=1,2$, $G[V(Q) \cap V(T_i)]$ is a tree where all vertices have odd degree.}

\2

{\em Proof of Claim~B:} Any vertex in $G$ with degree one must belong to the same tree, $T_j$, as its neighbour, as both $T_1$ and
$T_2$ have order at least two. By Claim~A, we therefore note that $G[V(Q) \cap V(T_i)]$ is a path of length $3n$ with a pendent edge 
attached to each non-endpoint of the path. This implies that $G[V(Q) \cap V(T_i)]$ is a tree where all vertices have odd degree (as
all degrees are either 1 or 3). This completes the proof of Claim~B.

\2

{\bf Claim~C:} {\em The instance $I$ is NAE-satisfiable.}

\2

{\em Proof of Claim~C:} Assume that the vertex $c_j$ belongs to $T_1$. 
First suppose that $|N_G(c_j) \cap V(T_1)|=0$. In this case $c_j$ has no neighbours in $T_1$, a contradiction, as $T_1$ is a tree with order at least two.
So $|N_G(c_j) \cap V(T_1)| \geq 1$.  Assume that $|N_G(c_j) \cap V(T_1)| \geq 2$. As $T_1$ is an induced tree in $G$, $c_j$ must have at least two 
neighbours, say $x$ and $y$, in $T_1$. However, by Claim~B, there exists a $(x,y)$-path in $T_1$ using only vertices from $V(Q)$, which 
implies that there is a cycle in $T_1$, a contradiction. Therefore $|N_G(c_j) \cap V(T_1)|=1.$

Analogously, we can show that $|N_G(c_j) \cap V(T_2)|=1$, whenever $c_j \in V(T_2)$.   
So each clause $C_j$ ($j \in [m]$) has either exactly one literal that is false (if $c_j \in V(T_1)$) 
or exactly one literal  that is true (if $c_j \in V(T_2)$). This implies that $I$ is NAE-satisfiable, which completes the proof of Claim~C and the theorem.

\section{Basics of parameterized complexity and Proof of Theorem \ref{ThmW}}\label{sec:W1}

In this section, we first provide necessary basic notions of parameterized complexity and then prove Theorem \ref{ThmW}.

\subsection{Basics of parameterized complexity}

\GG{
An instance of a parameterized problem $\Pi$
is a pair $(I,k)$ where $I$ is the {\em main part} and $k$ is the
{\em parameter}; the latter is usually a non-negative integer.  
A parameterized problem is
{\em fixed-parameter tractable} ({\sf FPT}) if there exists a computable function
$f$ such that instances $(I,k)$ can be solved in time $O(f(k)|{I}|^c)$
where $|I|$ denotes the size of~$I$ and $c$ is an absolute constant. The class of all fixed-parameter
tractable decision problems is called {{\sf FPT}}. 

Consider two parameterized problems $\Pi$ and $\Pi'$. We say that $\Pi$ has a {\em parameterized reduction} to $\Pi'$ if there are functions 
$k\mapsto k'$ and $k\mapsto k''$ from $\mathbb{N}$ to $\mathbb{N}$
and a function $(I,k)\mapsto (I',k')$ such that 

\begin{enumerate}
 \item  $(I,k)\mapsto (I',k')$ is computable in $k''(|I|+k)^{O(1)}$ time, and 
 \item $(I,k)$ is a yes-instance of $\Pi$ if and only if $(I',k')$ is a yes-instance of $\Pi'$.
\end{enumerate}

While {\sf FPT} is a parameterized complexity analog of {\sf P} in classic complexity theory, there are many parameterized hardness classes, forming a nested sequence of which {\sf FPT} is the first member: {\sf FPT}$\subseteq$ {\sf W}[1]$\subseteq$ {\sf W}[2 $]\subseteq \dots$.
It is well known that if the Exponential Time Hypothesis holds then ${\sf FPT} \ne {\sf W}[1]$.\footnote{The Exponential Time Hypothesis is a conjecture that there is no algorithm solving 3-CNF Satisfiability in time $2^{o(n)}$, where $n$ is the number of variables.}
Hence, {\sf W}[1] is generally viewed as a parameterized intractability class, which is an analog of {\sf NP} in classical complexity.  Consider the following two parameterized problems. In the {\sc Independent Set} problem parameterized by $k$, given a graph $G$ and a natural number $k$, we are to decide whether $G$ has an independent set with $k$ vertices. In the {\sc Dominating Set} problem parameterized by $k$, given a graph $G=(V,E)$ and a natural number $k$, we are to decide whether $G$ has a set $S$ of $k$ vertices such that every vertex in $V\setminus S$ is adjacent to some vertex in $S.$
A parameterized problem $\Pi$ is in {\sf W}[1] ({\sf W}[2], respectively) \AY{if there} is parameterized reduction from $\Pi$ to {\sc Independent Set} ({\sc Dominating Set}, respectively).
Thus, every {\sf W}[1]-hard problem $\Pi_1$ ( {\sf W}[2]-hard problem $\Pi_2$, respectively) is not `easier' than {\sc Independent Set} ({\sc Dominating Set}, respectively), i.e., {\sc Independent Set} ({\sc Dominating Set}, respectively) has a parameterized reduction to $\Pi_1$ ($\Pi_2,$ respectively).

For more information on parameterized algorithms and complexity, see recent books \cite{Cygan+,downey2013}.
}

\subsection{Proof of Theorem \ref{ThmW}}

\begin{definition}\label{def1}
Let $k$ be an integer and $G$ a graph of order $n-2$.  We now define $G'$, $H_1^G$, $H_2^G$ and $H_3^G$
as follows. (See Figure \ref{illustrationH} for illustration.)

\begin{description}
 \item[$G'$:] $G'$ is obtained from $G$ by adding to it two new isolated vertices.
Let $V(G')=\{v_1,v_2,\ldots,v_n\}$ such that $v_1$ and $v_k$ are the two isolated
vertices added to $G$ in order to create $G'$.

 \item[$H_1^G$:] Let $G_i'$  be a copy of $G'$, for $i \in \{1,2\}$, where
$V(G_i')=\{v_1^i,v_2^i,\ldots,v_n^i\}$ and $v_j^i$ is the copy of $v_j$ in $G'$.
Let $H_1^G$ be defined such that $V(H_1^G) = V(G'_1) \cup V(G'_2)$
and $N_{H_1^G}[v_j^1] = N_{H_1^G}[v_j^2] = N_{G'_1}[v_j] \cup  N_{G'_2}[v_j]$ for all
$j \in [n]$.

 \item[$H_2^G$:] Create $H_2^G$ from $H_1^G$ as follows. For for all $a,b \in [n]$ where 
$v_a^1 v_b^2 \not\in E(H_1^G)$ add the vertex $w_{a,b}$ and the 
 edges $v_a^1 w_{a,b}$ and $v_b^2 w_{a,b}$ to $H_1^G$.

 \item[$H_3^G$:] Recall that $v_1$ and $v_k$ are isolated vertices in $G'$.
Let $H_3^G$ be the graph obtained from $H_2^G$ by adding a pendent edge to all vertices in $H_2^G$ except
$v_1^1$ and $v_k^2$. Denote these pendent edges by $E_P^G$.
\end{description}
\end{definition}

In the construction of $H_2^G$ we could also add  $w_{a,b}$ and the                         
 edges $v_a^1 w_{a,b}$ and $v_b^2 w_{a,b}$ to $H_1^G$ for all $a,b \in [n]$ (without the condition that
$v_a^1 v_b^2 \not\in E(H_1^G)$) and the following theorem would still hold. 
We, however, chose to only add $w_{a,b}$ when $v_a^1 v_b^2 \not\in E(H_1^G)$ as this decreases the number of vertices and
edges in $H_2^G$ and $H_3^G$ and also makes it easier to depict these graphs.

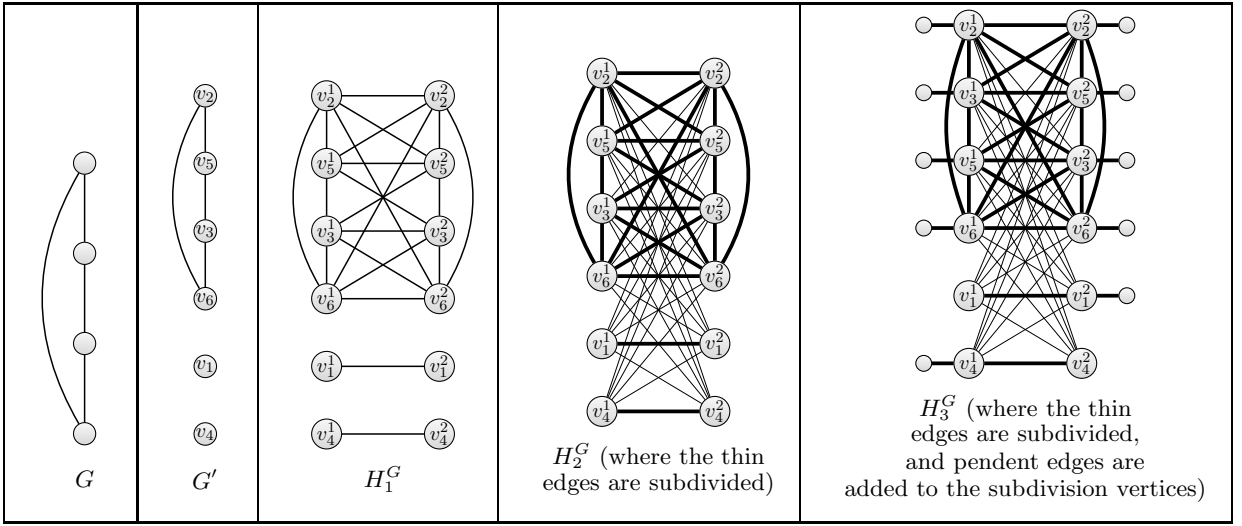
\begin{figure}[hbtp]
\begin{center}
\begin{tabular}{|c|c|c|c|c|} \hline
\tikzstyle{vertexL}=[circle,draw, top color=gray!5, bottom color=gray!30,minimum size=12pt, scale=0.7, inner sep=0.1pt]
\begin{tikzpicture}[scale=0.30]
\draw (0,-1) node {\mbox{ }};
\draw (2,14.5) node {\mbox{ }};
\draw (1,0) node {{\footnotesize $G$}};
\node (x1) at (1,14) [vertexL] {};
\node (x2) at (1,10) [vertexL] {};
\node (x3) at (1,6) [vertexL] {};
\node (x4) at (1,2) [vertexL] {};
\draw [line width=0.02cm] (x1) to (x2);
\draw [line width=0.02cm] (x2) to (x3);
\draw [line width=0.02cm] (x3) to (x4);
\draw [line width=0.02cm] (x1) to [out=240, in=120] (x4);
\end{tikzpicture} & 
\tikzstyle{vertexL}=[circle,draw, top color=gray!5, bottom color=gray!30,minimum size=12pt, scale=0.7, inner sep=0.1pt]
\begin{tikzpicture}[scale=0.3]
\draw (0,-2) node {\mbox{ }};
\draw (2,16.5) node {\mbox{ }};
\draw (1,-1) node {{\footnotesize $G'$}};
\node (x1) at (1,16) [vertexL] {$v_2$};
\node (x2) at (1,13) [vertexL] {$v_5$};
\node (x3) at (1,10) [vertexL] {$v_3$};
\node (x4) at (1,7) [vertexL] {$v_6$};
\node (x5) at (1,4) [vertexL] {$v_1$};
\node (x6) at (1,1) [vertexL] {$v_4$};
\draw [line width=0.02cm] (x1) to (x2);
\draw [line width=0.02cm] (x2) to (x3);
\draw [line width=0.02cm] (x3) to (x4);
\draw [line width=0.02cm] (x1) to [out=240, in=120] (x4);
\end{tikzpicture} & 
\tikzstyle{vertexL}=[circle,draw, top color=gray!5, bottom color=gray!30,minimum size=12pt, scale=0.7, inner sep=0.1pt]
\begin{tikzpicture}[scale=0.3]
\draw (0,-2) node {\mbox{ }};
\draw (7,16.5) node {\mbox{ }};
\draw (3.5,-1) node {{\footnotesize $H_1^G$}};
\node (x1) at (1,16) [vertexL] {$v_2^1$};
\node (x2) at (1,13) [vertexL] {$v_5^1$};
\node (x3) at (1,10) [vertexL] {$v_3^1$};
\node (x4) at (1,7) [vertexL] {$v_6^1$};
\node (x5) at (1,4) [vertexL] {$v_1^1$};
\node (x6) at (1,1) [vertexL] {$v_4^1$};
\node (y1) at (6,16) [vertexL] {$v_2^2$};
\node (y2) at (6,13) [vertexL] {$v_5^2$};
\node (y3) at (6,10) [vertexL] {$v_3^2$};
\node (y4) at (6,7) [vertexL] {$v_6^2$};
\node (y5) at (6,4) [vertexL] {$v_1^2$};
\node (y6) at (6,1) [vertexL] {$v_4^2$};

\draw [line width=0.02cm] (x1) to (x2);
\draw [line width=0.02cm] (x2) to (x3);
\draw [line width=0.02cm] (x3) to (x4);
\draw [line width=0.02cm] (x1) to [out=240, in=120] (x4);

\draw [line width=0.02cm] (y1) to (y2);
\draw [line width=0.02cm] (y2) to (y3);
\draw [line width=0.02cm] (y3) to (y4);
\draw [line width=0.02cm] (y1) to [out=300, in=60] (y4);

\draw [line width=0.02cm] (x1) to (y1);
\draw [line width=0.02cm] (x2) to (y2);
\draw [line width=0.02cm] (x3) to (y3);
\draw [line width=0.02cm] (x4) to (y4);
\draw [line width=0.02cm] (x5) to (y5);
\draw [line width=0.02cm] (x6) to (y6);
\draw [line width=0.02cm] (x1) to (y2);
\draw [line width=0.02cm] (x1) to (y4);
\draw [line width=0.02cm] (x2) to (y1);
\draw [line width=0.02cm] (x2) to (y3);
\draw [line width=0.02cm] (x3) to (y2);
\draw [line width=0.02cm] (x3) to (y4);
\draw [line width=0.02cm] (x4) to (y1);
\draw [line width=0.02cm] (x4) to (y3);
\end{tikzpicture} &
\tikzstyle{vertexL}=[circle,draw, top color=gray!5, bottom color=gray!30,minimum size=12pt, scale=0.7, inner sep=0.1pt]
\tikzstyle{vertexB}=[circle,draw, top color=gray!5, bottom color=gray!30,minimum size=5pt, scale=0.3, inner sep=0.05pt]

\begin{tikzpicture}[scale=0.3]
\draw (0,-3) node {\mbox{ }};
\draw (7,16.5) node {\mbox{ }};
\draw (3.5,-1) node {{\footnotesize $H_2^G$ (where the thin}};
\draw (3.5,-2.2) node {{\footnotesize edges are subdivided)}};

\node (x1) at (1,16) [vertexL] {$v_2^1$};
\node (x2) at (1,13) [vertexL] {$v_5^1$};
\node (x3) at (1,10) [vertexL] {$v_3^1$};
\node (x4) at (1,7) [vertexL] {$v_6^1$};
\node (x5) at (1,4) [vertexL] {$v_1^1$};
\node (x6) at (1,1) [vertexL] {$v_4^1$};
\node (y1) at (6,16) [vertexL] {$v_2^2$};
\node (y2) at (6,13) [vertexL] {$v_5^2$};
\node (y3) at (6,10) [vertexL] {$v_3^2$};
\node (y4) at (6,7) [vertexL] {$v_6^2$};
\node (y5) at (6,4) [vertexL] {$v_1^2$};
\node (y6) at (6,1) [vertexL] {$v_4^2$};

\draw [line width=0.05cm] (x1) to (x2);
\draw [line width=0.05cm] (x2) to (x3);
\draw [line width=0.05cm] (x3) to (x4);
\draw [line width=0.05cm] (x1) to [out=240, in=120] (x4);

\draw [line width=0.05cm] (y1) to (y2);
\draw [line width=0.05cm] (y2) to (y3);
\draw [line width=0.05cm] (y3) to (y4);
\draw [line width=0.05cm] (y1) to [out=300, in=60] (y4);

\draw [line width=0.05cm] (x1) to (y1);
\draw [line width=0.05cm] (x2) to (y2);
\draw [line width=0.05cm] (x3) to (y3);
\draw [line width=0.05cm] (x4) to (y4);
\draw [line width=0.05cm] (x5) to (y5);
\draw [line width=0.05cm] (x6) to (y6);

\draw [line width=0.05cm] (x1) to (y2);
\draw [line width=0.05cm] (x1) to (y4);
\draw [line width=0.05cm] (x2) to (y1);
\draw [line width=0.05cm] (x2) to (y3);
\draw [line width=0.05cm] (x3) to (y2);
\draw [line width=0.05cm] (x3) to (y4);
\draw [line width=0.05cm] (x4) to (y1);
\draw [line width=0.05cm] (x4) to (y3);

\draw [line width=0.01cm] (x1) to (y3);
\draw [line width=0.01cm] (x2) to (y4);
\draw [line width=0.01cm] (x3) to (y1);
\draw [line width=0.01cm] (x4) to (y2);

\draw [line width=0.01cm] (x1) to (y5);
\draw [line width=0.01cm] (x2) to (y5);
\draw [line width=0.01cm] (x3) to (y5);
\draw [line width=0.01cm] (x4) to (y5);
\draw [line width=0.01cm] (x6) to (y5);

\draw [line width=0.01cm] (x1) to (y6);
\draw [line width=0.01cm] (x2) to (y6);
\draw [line width=0.01cm] (x3) to (y6);
\draw [line width=0.01cm] (x4) to (y6);
\draw [line width=0.01cm] (x5) to (y6);

\draw [line width=0.01cm] (x5) to (y1);
\draw [line width=0.01cm] (x5) to (y2);
\draw [line width=0.01cm] (x5) to (y3);
\draw [line width=0.01cm] (x5) to (y4);

\draw [line width=0.01cm] (x6) to (y1);
\draw [line width=0.01cm] (x6) to (y2);
\draw [line width=0.01cm] (x6) to (y3);
\draw [line width=0.01cm] (x6) to (y4);
\end{tikzpicture} & 
\tikzstyle{vertexL}=[circle,draw, top color=gray!5, bottom color=gray!30,minimum size=12pt, scale=0.7, inner sep=0.1pt]
\tikzstyle{vertexB}=[circle,draw, top color=gray!5, bottom color=gray!30,minimum size=10pt, scale=0.6, inner sep=0.1pt]

\begin{tikzpicture}[scale=0.3]
\draw (0,-3) node {\mbox{ }};
\draw (7,16.5) node {\mbox{ }};
\draw (3.5,-1) node {{\footnotesize $H_3^G$ (where the thin}};
\draw (3.5,-2.2) node {{\footnotesize edges are subdivided,}};
\draw (3.5,-3.4) node {{\footnotesize and pendent edges are}};
\draw (3.5,-4.6) node {{\footnotesize added to the subdivision vertices)}};

\node (x1) at (1,16) [vertexL] {$v_2^1$};
\node (x2) at (1,13) [vertexL] {$v_3^1$};
\node (x3) at (1,10) [vertexL] {$v_5^1$};
\node (x4) at (1,7) [vertexL] {$v_6^1$};
\node (x5) at (1,4) [vertexL] {$v_1^1$};
\node (x6) at (1,1) [vertexL] {$v_4^1$};
\node (y1) at (6,16) [vertexL] {$v_2^2$};
\node (y2) at (6,13) [vertexL] {$v_5^2$};
\node (y3) at (6,10) [vertexL] {$v_3^2$};
\node (y4) at (6,7) [vertexL] {$v_6^2$};
\node (y5) at (6,4) [vertexL] {$v_1^2$};
\node (y6) at (6,1) [vertexL] {$v_4^2$};

\node (a1) at (-1,16) [vertexB] {};
\node (a2) at (-1,13) [vertexB] {};
\node (a3) at (-1,10) [vertexB] {};
\node (a4) at (-1,7) [vertexB] {};
\node (a6) at (-1,1) [vertexB] {};
\node (b1) at (8,16) [vertexB] {};
\node (b2) at (8,13) [vertexB] {};
\node (b3) at (8,10) [vertexB] {};
\node (b4) at (8,7) [vertexB] {};
\node (b5) at (8,4) [vertexB] {};

\draw [line width=0.05cm] (a1) to (x1);
\draw [line width=0.05cm] (a2) to (x2);
\draw [line width=0.05cm] (a3) to (x3);
\draw [line width=0.05cm] (a4) to (x4);
\draw [line width=0.05cm] (a6) to (x6);

\draw [line width=0.05cm] (b1) to (y1);
\draw [line width=0.05cm] (b2) to (y2);
\draw [line width=0.05cm] (b3) to (y3);
\draw [line width=0.05cm] (b4) to (y4);
\draw [line width=0.05cm] (b5) to (y5);

\draw [line width=0.05cm] (x1) to (x2);
\draw [line width=0.05cm] (x2) to (x3);
\draw [line width=0.05cm] (x3) to (x4);
\draw [line width=0.05cm] (x1) to [out=250, in=110] (x4);

\draw [line width=0.05cm] (y1) to (y2);
\draw [line width=0.05cm] (y2) to (y3);
\draw [line width=0.05cm] (y3) to (y4);
\draw [line width=0.05cm] (y1) to [out=290, in=70] (y4);

\draw [line width=0.05cm] (x1) to (y1);
\draw [line width=0.05cm] (x2) to (y2);
\draw [line width=0.05cm] (x3) to (y3);
\draw [line width=0.05cm] (x4) to (y4);
\draw [line width=0.05cm] (x5) to (y5);
\draw [line width=0.05cm] (x6) to (y6);

\draw [line width=0.05cm] (x1) to (y2);
\draw [line width=0.05cm] (x1) to (y4);
\draw [line width=0.05cm] (x2) to (y1);
\draw [line width=0.05cm] (x2) to (y3);
\draw [line width=0.05cm] (x3) to (y2);
\draw [line width=0.05cm] (x3) to (y4);
\draw [line width=0.05cm] (x4) to (y1);
\draw [line width=0.05cm] (x4) to (y3);

\draw [line width=0.01cm] (x1) to (y3);
\draw [line width=0.01cm] (x2) to (y4);
\draw [line width=0.01cm] (x3) to (y1);
\draw [line width=0.01cm] (x4) to (y2);

\draw [line width=0.01cm] (x1) to (y5);
\draw [line width=0.01cm] (x2) to (y5);
\draw [line width=0.01cm] (x3) to (y5);
\draw [line width=0.01cm] (x4) to (y5);
\draw [line width=0.01cm] (x6) to (y5);

\draw [line width=0.01cm] (x1) to (y6);
\draw [line width=0.01cm] (x2) to (y6);
\draw [line width=0.01cm] (x3) to (y6);
\draw [line width=0.01cm] (x4) to (y6);
\draw [line width=0.01cm] (x5) to (y6);

\draw [line width=0.01cm] (x5) to (y1);
\draw [line width=0.01cm] (x5) to (y2);
\draw [line width=0.01cm] (x5) to (y3);
\draw [line width=0.01cm] (x5) to (y4);

\draw [line width=0.01cm] (x6) to (y1);
\draw [line width=0.01cm] (x6) to (y2);
\draw [line width=0.01cm] (x6) to (y3);
\draw [line width=0.01cm] (x6) to (y4);

\end{tikzpicture}  \\ \hline
\end{tabular}
\end{center}
\caption{An illustration of $G$, $G'$, $H_1^G$, $H_2^G$ and $H_3^G$ when $G$ is a $4$-cycle; $k=4$. } \label{illustrationH}
\end{figure}

\AY{
Theorem \ref{ThmW} follows from the fact that {\sc Independent Set}
is {\sf W}[1]-complete \cite{Cygan+,downey2013} and that the mapping
of $G$ with parameter $k-2$ to $H_3^G$ with parameter $3k-3$ is a
parameterized reduction from {\sc Independent Set} to {\sc Perfect
Forest Above Perfect Matching}.
The proof of equivalence between (a) and (d) in the following theorem
and the definition of $H_3^G$ show that the mapping is indeed a
parameterized reduction.

See furthermore Figure~\ref{illustrationHii} for an example that illustrates the proof of Theorem~\ref{thmX}.
}

\begin{theorem} \label{thmX}
The following statements are equivalent. 

\begin{description}
 \item[(a)] $G$ contains an independent set of size $k-2$.
 \item[(b)] $G'$ contains an independent set of size $k$.
 \item[(c)] $H_2^G$ contains an induced $(v_1^1,v_k^2)$-path in $H_2^G$ of length $3k-2$.
 \item[(d)] $H_3^G$ contains a $0$-perfect forest with at least $\frac{|V(H_3^G(a,b))|}{2} + 3k-3$  edges.
\end{description}
\end{theorem}

\begin{proof}
Clearly (a) and (b) are equivalent as $v_1$ and $v_k$ are isolated vertices in $G'$ and $G = G' - \{v_1,v_k\}$.

\2

We will now show that (b) and (c) are equivalent.
First assume that there is an independent set of size $k$ in $G'$. Without loss of generality we may assume that 
$\{v_1,v_2,\ldots,v_k\}$ is an independent set in $G'$ as any maximum independent set in $G'$ contains $v_1$ and $v_k$. 
The following is an induced $(v_1^1,v_k^2)$-path of length $3k-2$ in $H_2^G$.

\[
 v_1^1   \SpL{}
 v_1^2   \SpL{}
 w_{2,1} \SpL{}
 v_2^1   \SpL{}
 v_2^2   \SpL{}
 w_{3,2} \SpL{}
 v_3^1   \SpL{}
 v_3^2   \SpL{}
 w_{4,3} \SpL{} \cdots  \SpL{}
 v_{k-1}^2  \SpL{}
 w_{k,k-1}  \SpL{}
 v_k^1   \SpL{}
 v_k^2   
\]

Conversely assume that $P=p_1 p_2 p_3 \ldots p_{3k-2}$ is an induced $(v_1^1,v_k^2)$-path in $H_2^G$ 
of length $3k-2$. Let $R$ be a maximum subset of \GG{$V(P) \cap V(H_1^G)$} 
such that each pair of vertices in $R$ are non-adjacent in $P$ and $p_1 \in R$.
As $R \subseteq V(H_1^G)$ we may name the vertices in $R$ as follows $R=\{v_{a_1}^{b_1}, v_{a_2}^{b_1}, \ldots, v_{a_r}^{b_r}\}$,
where $r = |R|$.   Let $R^* = \{v_{a_1}, v_{a_2}, \ldots, v_{a_r}\}$.
As $P$ is an induced path and no adjacent vertices in $P$ belong to $R$ we note that $R$ is an independent set in $H_2^G$
and therefore $a_1, a_2, \ldots, a_r$ are $r$ distinct indices and $R^*$ is an independent set in $G'$ of size $r$. 

Let $R = \{p_{q_1},p_{q_2},\ldots, p_{q_r}\}$, where $1=q_1 < q_2 < \cdots < q_r$. 
Note that $q_2 \leq 4$, as $q_2=3$ if $p_3 \in V(H_1^G)$ and $q_2=4$ otherwise.
Analogously $q_3 \leq 7$, $q_4 \leq 10$, etc.. This implies that $q_i \leq 3i-2$ for all $i \in [r]$ and therefore
$r \geq k$ (as $q_k \leq 3k-2$).  This implies that $R^*$ is an independent set in $G$ of size at least $k$, which 
implies that (b) and (c) are equivalent.

\2

We will now show that (c) and (d) are equivalent.
Assume that $F$ is a $0$-perfect forest in $H_3^G$ with at least $\frac{|V(H_3^G)|}{2} + 3k-3$ edges. 
Clearly all the pendent edges, $E_P^G$, added to $H_2^G$ in order to create $H_3^G$ belong to $F$. 
Note that $|E_P^G|=\frac{|V(H_3^G)|}{2} -1$.
Furthermore $F-E_P^G$ is an induced forest in $H_2^G$ 
where all vertices have even degree except $v_1^1$ and $v_k^2$ which both have odd degree 
(as they are the only vertices not adjacent to a pendent edge from $E_P^G$). 
This implies that $F-E_P^G$ is an induced $(v_1^1,v_k^2)$-path, $P$, and a number of isolated vertices.
As $|E(F)| \geq \frac{|V(H_3^G)|}{2} + 3k-3$ we note that $|E(P)| = |E(F)|-|E_P^Q| \geq 3k-2$.
Therefore if (d) holds then (c) holds.

Conversely, assume that (c) holds and there exists an induced $(v_1^1,v_k^2)$-path, $P'$, in $H_2^G$ of length at least $3k-2$.
Adding all pendent edges $E_P^Q$ to $P'$ gives us a $0$-perfect forest with $\frac{|V(H_3^G)|}{2} - 1 + |E(P')|
\geq \frac{|V(H_3^G)|}{2} + 3k-3$ edges. Therefore if (c) holds then (d) holds. 
Therefore (a), (b), (c) and (d) either all hold of none of them holds.
\end{proof}


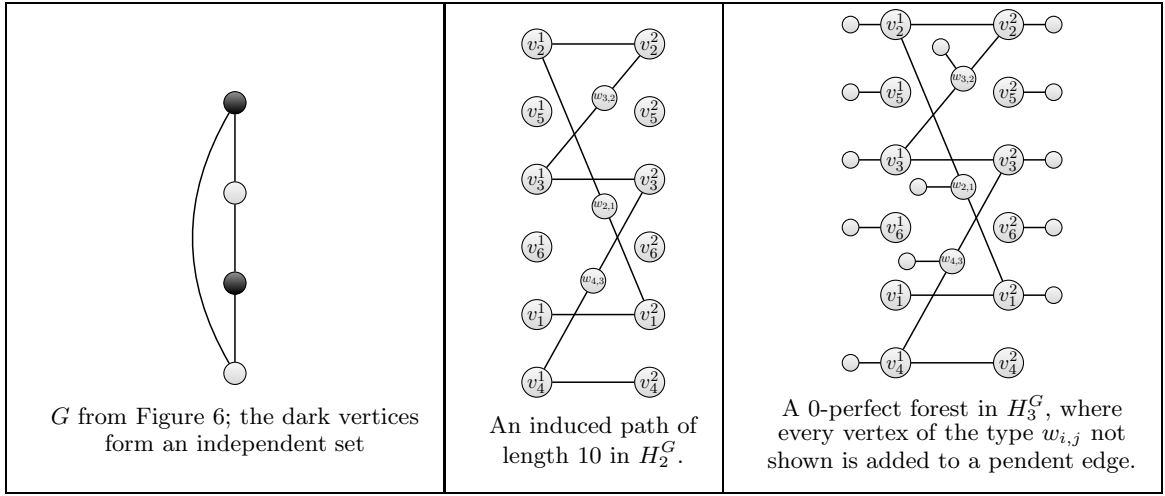
\begin{figure}[hbtp]
\begin{center}
\begin{tabular}{|c|c|c|} \hline
\tikzstyle{vertexI}=[circle,draw, top color=black!40, bottom color=black!99,minimum size=12pt, scale=0.7, inner sep=0.1pt]
\tikzstyle{vertexL}=[circle,draw, top color=gray!5, bottom color=gray!30,minimum size=12pt, scale=0.7, inner sep=0.1pt]
\begin{tikzpicture}[scale=0.30]
\draw (0,-1) node {\mbox{ }};
\draw (2,14.5) node {\mbox{ }};
\draw (1,0) node {{\footnotesize $G$ from Figure \ref{illustrationH}; the dark vertices}};
\draw (1,-1.2) node {{\footnotesize  form an independent set}};
\draw (1,-2.4) node {{\footnotesize }};

\node (x1) at (1,14) [vertexI] {};
\node (x2) at (1,10) [vertexL] {};
\node (x3) at (1,6) [vertexI] {};
\node (x4) at (1,2) [vertexL] {};
\draw [line width=0.02cm] (x1) to (x2);
\draw [line width=0.02cm] (x2) to (x3);
\draw [line width=0.02cm] (x3) to (x4);
\draw [line width=0.02cm] (x1) to [out=240, in=120] (x4);
\end{tikzpicture} & 
\tikzstyle{vertexL}=[circle,draw, top color=gray!5, bottom color=gray!30,minimum size=14pt, scale=0.7, inner sep=0.1pt]
\tikzstyle{vertexB}=[circle,draw, top color=gray!5, bottom color=gray!30,minimum size=14pt, scale=0.45, inner sep=0.05pt]
\begin{tikzpicture}[scale=0.3]
\draw (0,-3) node {\mbox{ }};
\draw (7,16.5) node {\mbox{ }};
\draw (3.5,-1) node {{\footnotesize An induced path of}};
\draw (3.5,-2.2) node {{\footnotesize length 10 in $H_2^G$. }};

\node (w63) at (3.5,5.5) [vertexB] {$w_{4,3}$};
\node (w15) at (4,8.8) [vertexB] {$w_{2,1}$};
\node (w31) at (4,13.6) [vertexB] {$w_{3,2}$};

\node (x1) at (1,16) [vertexL] {$v_2^1$};
\node (x2) at (1,13) [vertexL] {$v_5^1$};
\node (x3) at (1,10) [vertexL] {$v_3^1$};
\node (x4) at (1,7) [vertexL] {$v_6^1$};
\node (x5) at (1,4) [vertexL] {$v_1^1$};
\node (x6) at (1,1) [vertexL] {$v_4^1$};
\node (y1) at (6,16) [vertexL] {$v_2^2$};
\node (y2) at (6,13) [vertexL] {$v_5^2$};
\node (y3) at (6,10) [vertexL] {$v_3^2$};
\node (y4) at (6,7) [vertexL] {$v_6^2$};
\node (y5) at (6,4) [vertexL] {$v_1^2$};
\node (y6) at (6,1) [vertexL] {$v_4^2$};

\draw [line width=0.02cm] (x1) to (y1); 
\draw [line width=0.02cm] (x3) to (y3); 
\draw [line width=0.02cm] (x5) to (y5);  
\draw [line width=0.02cm] (x6) to (y6);  

\draw [line width=0.02cm] (x3) to (w31);  \draw [line width=0.02cm] (w31) to (y1); 
\draw [line width=0.02cm] (x1) to (w15);  \draw [line width=0.02cm] (w15) to (y5); 
\draw [line width=0.02cm] (x6) to (w63);  \draw [line width=0.02cm] (w63) to (y3); 
\end{tikzpicture} & 
\tikzstyle{vertexL}=[circle,draw, top color=gray!5, bottom color=gray!30,minimum size=14pt, scale=0.7, inner sep=0.1pt]
\tikzstyle{vertexB}=[circle,draw, top color=gray!5, bottom color=gray!30,minimum size=14pt, scale=0.45, inner sep=0.05pt]
\begin{tikzpicture}[scale=0.3]
\draw (0,-3) node {\mbox{ }};
\draw (7,16.5) node {\mbox{ }};
\draw (3.5,-1) node {{\footnotesize A $0$-perfect forest in $H_3^G$, where}};
\draw (3.5,-2.2) node {{\footnotesize  every vertex of the type $w_{i,j}$ not}};
\draw (3.5,-3.4) node {{\footnotesize shown is added to a pendent edge.}};

\node (w63) at (3.5,5.5) [vertexB] {$w_{4,3}$};
\node (w15) at (4,8.8) [vertexB] {$w_{2,1}$};
\node (w31) at (4,13.6) [vertexB] {$w_{3,2}$};

\node (p63) at (1.5,5.5) [vertexB] {};
\node (p15) at (2,8.8) [vertexB] {};
\node (p31) at (3,15) [vertexB] {};

\draw [line width=0.02cm] (w63) to (p63);
\draw [line width=0.02cm] (w15) to (p15);
\draw [line width=0.02cm] (w31) to (p31);

\node (x1) at (1,16) [vertexL] {$v_2^1$};
\node (x2) at (1,13) [vertexL] {$v_5^1$};
\node (x3) at (1,10) [vertexL] {$v_3^1$};
\node (x4) at (1,7) [vertexL] {$v_6^1$};
\node (x5) at (1,4) [vertexL] {$v_1^1$};
\node (x6) at (1,1) [vertexL] {$v_4^1$};
\node (y1) at (6,16) [vertexL] {$v_2^2$};
\node (y2) at (6,13) [vertexL] {$v_5^2$};
\node (y3) at (6,10) [vertexL] {$v_3^2$};
\node (y4) at (6,7) [vertexL] {$v_6^2$};
\node (y5) at (6,4) [vertexL] {$v_1^2$};
\node (y6) at (6,1) [vertexL] {$v_4^2$};

\draw [line width=0.02cm] (x1) to (y1); 
\draw [line width=0.02cm] (x3) to (y3); 
\draw [line width=0.02cm] (x5) to (y5);  
\draw [line width=0.02cm] (x6) to (y6);  
\draw [line width=0.02cm] (x3) to (w31);  \draw [line width=0.02cm] (w31) to (y1); 
\draw [line width=0.02cm] (x1) to (w15);  \draw [line width=0.02cm] (w15) to (y5); 
\draw [line width=0.02cm] (x6) to (w63);  \draw [line width=0.02cm] (w63) to (y3); 

\node (a1) at (-1,16) [vertexB] {};
\node (a2) at (-1,13) [vertexB] {};
\node (a3) at (-1,10) [vertexB] {};
\node (a4) at (-1,7) [vertexB] {};
\node (a6) at (-1,1) [vertexB] {};
\node (b1) at (8,16) [vertexB] {};
\node (b2) at (8,13) [vertexB] {};
\node (b3) at (8,10) [vertexB] {};
\node (b4) at (8,7) [vertexB] {};
\node (b5) at (8,4) [vertexB] {};

\draw [line width=0.02cm] (a1) to (x1);
\draw [line width=0.02cm] (a2) to (x2);
\draw [line width=0.02cm] (a3) to (x3);
\draw [line width=0.02cm] (a4) to (x4);
\draw [line width=0.02cm] (a6) to (x6);

\draw [line width=0.02cm] (b1) to (y1);
\draw [line width=0.02cm] (b2) to (y2);
\draw [line width=0.02cm] (b3) to (y3);
\draw [line width=0.02cm] (b4) to (y4);
\draw [line width=0.02cm] (b5) to (y5);

\end{tikzpicture}  \\ \hline
\end{tabular}
\end{center}
\caption{An illustration of an independent set of size two in $G$, an induced $(v_1^1,v_4^2)$-path in 
$H_2^G$ and a $0$-perfect forest in $H_3^G$. Note that not all vertices in $H_3^G$ are depicted and
every vertex of the type $w_{i,j}$ which is not in the picture is added to a pendent edge from $E_P^G$.
However, these edges are, for the clarity of the picture, not shown.} \label{illustrationHii}
\end{figure}

\section{Proof of Theorem \ref{ZeroPerfectForestContainingEdge}}\label{sec:ince}

To prove Theorem \ref{ZeroPerfectForestContainingEdge}, we will use the following result.
The proof of Theorem \ref{inducedCycle} follows the same approach as the proof that it is 
{\sf NP}-hard to determine whether there is an induced cycle of odd length through a prescribed vertex, given in \cite{DB91} by Bienstock.

\begin{theorem} \label{inducedCycle} 
It is {\sf NP}-hard to determine whether a graph contains an induced cycle through two given edges.
\end{theorem} 
\begin{proof}
We will reduce from the well-known {\sf NP}-hard 3-SAT problem. Let $I$ be an instance of 3-SAT with variables
$v_1,v_2,\ldots,v_n$ and clauses $C_1,C_2,\ldots,C_m$. We will construct a graph $G$ with two edges $e_1$ and $e_2$ such that
$G$ contains an induced cycle, $C$, with $e_1,e_2 \in E(C)$ if and only if $I$ is satisfiable.

We first create a gadget $H_i$ for each $i=1,2,\ldots,n$ as follows. 
Let $V(H_i) = \{x_1^i,w_1^i,\overline{w}_1^i,y_1^i, x_2^i,w_2^i,\overline{w}_2^i,y_2^i\}$ and let $E(H_i)$ be defined as follows
(see Figure~\ref{gadgetHi}):
$
E(H_i) = \{ x_j^i w_j^i, x_j^i \overline{w}_j^i, w_j^i y_j^i, \overline{w}_j^i y_j^i \; | \; j=1,2\} 
\cup \{w_1^i \overline{w}_2^i, \overline{w}_1^i w_2^i \}.
$

\begin{figure}[hbtp]
\input{pic4.tex}
\caption{Two different drawings of $H_i$.} \label{gadgetHi}
\end{figure}
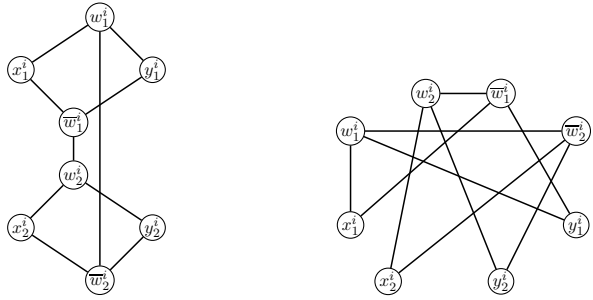

Let $Q_1$ be the graph with vertex set $V(Q_1)=\cup_{i=1}^n V(H_i)$ and the following edge set, where
$e_1 = x_1^1 x_2^1$.
\[
E(Q_1)= \{ e_1 \} \cup \{ y_1^i x_1^{i+1}, y_2^i x_2^{i+1} \; | \; i=1,2,\ldots,n-1 \} \cup \left(\cup_{i=1}^n E(H_i) \right)
\]
This is the part of the graph $G$ representing the variables of $I$. We now create a gadget $R_i$ for each $i=1,2,\ldots,m$
that will represent the clauses of $I$. Let $V(R_i) = \{ a^i, c_1^i, c_2^i, c_3^i, b^i \}$ and let $E(R_i)$ contain all 
edges between $\{a^i,b^i\}$  and $\{c_1^i, c_2^i, c_3^i \}$ (i.e. $R_i$ is isomorphic to $K_{2,3}$).
Let $Q_2$  be the graph with vertex set $V(Q_2)=\cup_{i=1}^m V(R_i)$ and the following edge set.
\[
E(Q_1)= \{ b^i a^{i+1} \; | \; i=1,2,\ldots,m-1 \} \cup \left(\cup_{i=1}^m E(R_i) \right)
\]
We now create $G$ by letting $V(G) = V(Q_1) \cup V(Q_2)$ and $E(G) = E(Q_1) \cup E(Q_2) \cup \{b^m y_1^n, y_2^n a^1 \}$.
Recall that $e_1 = x_1^1 x_2^1$ and let $e_2 = b^m y_1^n$. Se Figure~\ref{exampleII} for an illustration of $G$.
We will now show that $G$ contains an induced cycle, $C$, with $e_1,e_2 \in E(C)$ if and only if $I$ is satisfiable.

\begin{figure}[hbtp]
\input{pic5.tex}
\caption{The graph $G$ if $I=(\overline{v_1} \vee v_2 \vee v_3) \wedge (\overline{v}_2 \vee \overline{v_3} \vee v_4)$.} \label{exampleII}
\end{figure}
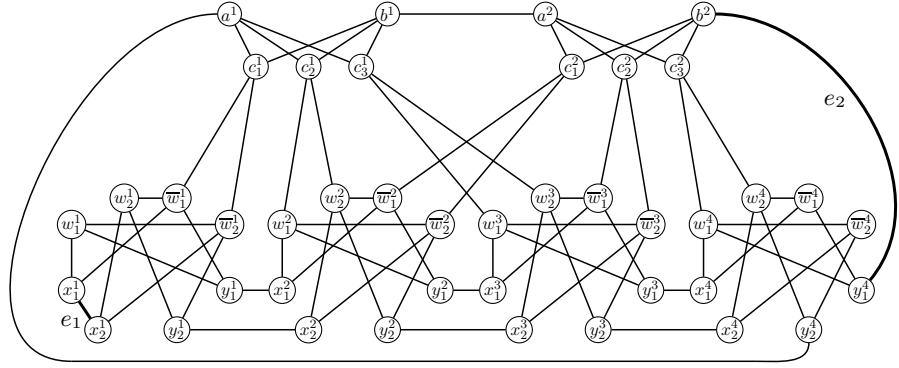

First assume that $I$ is satisfiable and we are given a truth assignment to all the variables that satisfies $I$. 
If $v_i = true$ then we add the edges in the path $x_1^i \overline{w}_1^i y_1^i$ and the path $x_2^i \overline{w}_2^i y_2^i$
to our solution $S$. If $v_i = false$ then we add the edges in the path $x_1^i w_1^i y_1^i$ and the path $x_2^i w_2^i y_2^i$
to $S$. We then add all edges of the form $y_1^i x_1^{i+1}$ and $y_2^i x_2^{i+1}$ to $S$, for $i \in [n-1]$. We then 
add the edges $e_1, e_2, y_2^n a^1$ to $S$. For all $j \in [m]$ we let $k_j \in [3]$ be defined such that 
the $k_j$'th literal in $C_j$ is satisfied, and we then add the edges in the path $a^j c_{k_j}^j b^j$ to $S$.
Finally we add the edge $b^j a^{j+1}$ for $j \in [m-1]$ to $S$. We have now obtained an induced cycle, $S$, in $G$ 
containing both $e_1$ and $e_2$ (see Figure~\ref{exampleIIsol} for an illustration).

\begin{figure}[hbtp]
\input{pic6.tex}
\caption{An induced cycle in $G$, containing $e_1$ and $e_2$, when considering the solution $v_1=v_3=v_4=true$ and $v_2=false$ for $I=(\overline{v_1} \vee v_2 \vee v_3) \wedge (\overline{v}_2 \vee \overline{v_3} \vee v_4)$ (as literal $v_3$ is true in clause $1$ and literal $\overline{v}_2$ is true in clause $2$).} \label{exampleIIsol}
\end{figure}
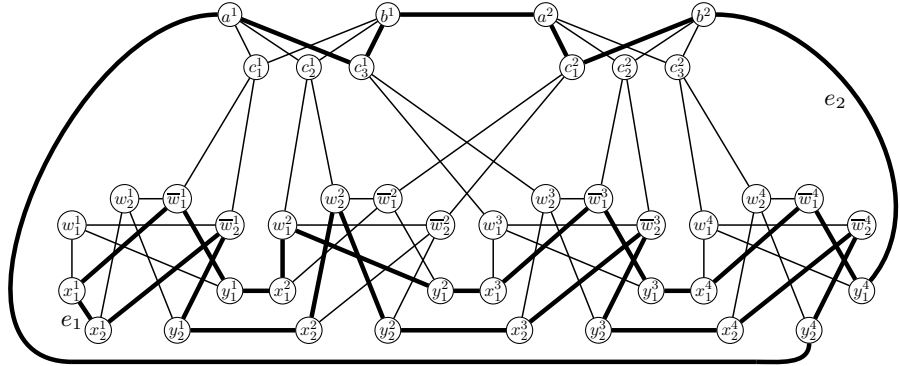

Conversely assume that there exists an induced cycle, $C$, in $G$, containing $e_1$ and $e_2$. We will show that $I$ must be satisfiable. We first prove Claim~A below.

\2

{\bf Claim~A:} {\em One of the following two cases occur for all $i \in [n]$.

\begin{description}
 \item[Case 1:] $x_1^i w_1^i y_1^i$ and $x_2^i w_2^i y_2^i$ are subpaths in $C$, or
 \item[Case 2:] $x_1^i \overline{w}_1^i y_1^i$ and $x_2^i \overline{w}_2^i y_2^i$ are subpaths in $C$.
\end{description}
}

\2

{\em Proof of Claim~A:} The proof is by induction on $i$. Let $i \in [n]$ be arbitrary and 
assume that the above holds for all smaller values of $i$ (the base case $i=0$ holds vacuously).
Let $P_1 = C[x_1^1,x_1^i]$ (if $i=1$ $V(P_1)=\{x_1^1\}$) and let $P_2 = C[x_2^1,x_2^i]$.
Note that, by induction, $P_1$ and $P_2$ both exist and they are both paths passing through $H_1, H_2, \ldots, H_{i-1}$.
Note that $w_1^i$ or $\overline{w}_1^i$ is a neighbour of $x_1^i$ on $C$, due to the existence of $P_1$.

First consider the case when $x_1^i w_1^i \in E(C)$. If $x_2^i \overline{w}_2^i \in E(C)$, then 
$w_1^i \overline{w}_2^i \in E(C)$ (as otherwise $w_1^i \overline{w}_2^i$ would be a chord in $C$, a contradiction
to $C$ being induced), a contradiction to $e_2 \in E(C)$ (as in this case $V(C) = V(P_1) \cup V(P_2) \cup \{w_1^i, \overline{w}_2^i\}$).
Therefore $x_2^i \overline{w}_2^i \not\in E(C)$, which implies that $x_2^i w_2^i \in E(C)$.

For the sake of contradiction, assume that $w_1^i y_1^i \not\in E(C)$, which implies that 
$w_1^i c_k^j \in E(C)$, for some $ k\in [3]$ and $j \in [m]$. However we must now have 
$c_k^j w_2^j \in E(C)$ (as otherwise $c_k^j w_2^j$ would be a chord in $C$), a contradiction to $e_2 \in E(C)$
(as in this case $V(C) = V(P_1) \cup V(P_2) \cup \{w_1^i, w_2^i, c_k^j\}$)). Therefore,
$w_1^i y_1^i \in E(C)$. Analogously, we can show that $w_2^i y_2^i \in E(C)$. This implies that
we are in Case~1 above, whenever $x_1^i w_1^i \in E(C)$.

If $x_1^i w_1^i \not\in E(C)$ then we must have $x_1^i \overline{w}_1^i \in E(C)$. In this case we can analogously 
to above show that we are in Case 2, which completes the proof of Claim~A.

\2

We now return to the proof of the theorem. By Claim~A we note that both $C[x_1^1,y_1^n]$ and
$C[x_2^1,y_2^n]$ pass through $H_1, H_2, \ldots, H_{n}$ in that order. For each $i$ we also note that 
$V(C) \cap \{w_1^i,w_2^i,\overline{w}_1^i,\overline{w}_2^i\}$ is either $\{w_1^i,w_2^i\}$ or 
$\{\overline{w}_1^i,\overline{w}_2^i\}$, as if we are in Case 1 in Claim~A the vertices 
$\overline{w}_1^i,\overline{w}_2^i$ cannot belong to $C$ (as $C$ would then have a chord) and if we
are in Case 2 then $w_1^i,w_2^i$ cannot belong to $C$. 

Clearly $b^m y_1^n$ ($=e_2$) and $y_2^n a^1$ are both edges of $C$ and $C[a^1,b^m]$ is a path passing through
$R_1,R_2,\ldots,R_m$ in that order. So for each $j \in [m]$ there exists exactly one $k_j \in [3]$ such that 
$c_{k_j}^j \in V(C)$. If we let $v_i$ be true if and only if neither $w_1^i$ or $w_2^i$ belong to $C$, then we
note that the $k_j$'th literal in the $j$'th clause must be satisfied.  So $I$ is satisfiable, which completes the proof.
\end{proof}

\thsix*
\begin{proof}
Let $G$ be a graph and let $e_1 = u_1v_1$ and $e_2 = u_2 v_2$ be distinct edges of $G$. 
We will construct an auxillary graph $H$ with an edge $e_2' \in E(H)$, such that $H$ 
contains a $0$-perfect forest containing $e_2'$ if and only if $G$ contains an induced cycle, $C$, such that
$e_1,e_2 \in E(C)$. This will complete the proof by Theorem~\ref{inducedCycle}.

Let $H$ be obtained from $G$ by adding a pendent edge to each vertex in $V(G) \setminus \{u_1,v_1\}$ and deleting 
the edge $e_1$. Let $E_P$ denote the set of all the pendent edges we just added to $G$. 
 Let $e_2' = u_2 v_2$ and note that $e_2' \in E(H)$. This completes the construction of $H$.

Assume that there exists an induced cycle, $C$, in $G$ such that 
$e_1,e_2 \in E(C)$. Let $E' = E_P \cup E(C) \setminus e_1$.
Note that the edges in $E'$ induce a $0$-perfect forest in $H$ containing the edge $e_2'$.

Conversely assume that there is a $0$-perfect forest, $F$, in $H$ containing $e_2'$. Clearly $F$ contains 
all edges in $E_p$ as each pendent edge is incident with a vertex of degree one.
Let $Q$ be the subgraph of $H$ induced by the edges in $E(F)\setminus E_P$.
Note that $Q$ is a perfect forest where $u_1$ and $v_1$ have odd
degree and all other vertices have even degree. 
As $Q$ is a perfect forest all components are induced trees, and as $u_1$ and $v_1$ are the only vertices of
odd degree, this implies that $Q$ is an induced path between $u_1$ and $v_1$. 
Adding the edge $e_1$ to $Q$ gives us an induced cycle in $G$ containing both $e_1$ and $e_2$ (as $e_2' \in E(F)$).

Therefore we have proven that $H$ contains a $0$-perfect forest containing $e_2'$ if and only if 
$G$ contains an induced cycle, $C$, such that $e_1,e_2 \in E(C)$, as desired.
\end{proof}

\section{Proof of Theorem \ref{FparityForestAvoidingEdge}}\label{sec:noe}


Let $G$ be a graph and $e = uv$  an edge of $G.$ 
Let  $f\colon V(G) \rightarrow \{0,1\}$ be an even-sum function. 
Our polynomial-time algorithm will follow from the four claims proved below. At the end of the proof, we briefly discuss
how the claims are used in the algorithm. 

\2

{\bf Claim A:} {\em Suppose that $G$ contains a cut-vertex $x$, which may or may not belong to $\{u,v\}.$
Let $C$ be the component in $G-x$ intersecting $\{u,v\}$ (there is exactly one such component as
$uv \in E(G)$) and let $G' = G[V(C) \cup \{x\}]$. Let $f'(w)=f(w)$ for all $w \in V(C)$ and define
$f'(x) \in \{0,1\}$ such that $\sum_{z \in V(G_i)} f'(z)$ is even.
Then $G$ has an $f$-parity perfect forest not containing $e$ if and only if
$G'$ has an $f'$-parity perfect forest not containing $e$.
 }

\2

{\em Proof of Claim~A:} Let $G$ contain a cut-vertex $x$ and 
let $C_1,C_2, \ldots, C_k$ be the components in $G-x$.
Without loss of generality,  assume that $C_1$ is the component intersecting $\{u,v\}$.
Let $G_i = G[V(C_i)\cup \{x\}]$ for all $i \in [k]$. 

For each $i \in [k]$ we will let $f_i :V(G_i) \rightarrow \{0,1\}$ be defined such that
$f_i(w)=f(w)$ for all $w \in V(C_i)$ and $\sum_{z \in V(G_i)} f_i(z)$ is even (this defines the value of
$f_i(x)$).
We will show that $G$ has an $f$-parity perfect forest not containing $e$ if and only if 
$G_1$ has an $f_1$-parity perfect forest not containing $e$, which will complete the proof of Claim~A.

First assume that $G_1$ has an $f_1$-parity perfect forest $F_1$ not containing $e$.
By Theorem~\ref{th:fpar} there exists an $f_i$-parity perfect forest, $F_i$, in $G_i$ for all $i=2,3,\ldots,k$.
Now $F_1 \cup F_2 \cup \cdots \cup F_k$ is an $f$-parity perfect forest of $G$ not containing $e$, as desired.

Conversely assume that $G$ has an $f$-parity perfect forest $F$ not containing $e$.
If we restrict $F$ to $V(G_1)$, then we obtain an $f_1$-parity perfect forest of $G_1$ not containing $e$.

\2

{\bf Claim~B:} {\em If $G$ is $2$-connected and $f(u)=0$ or $f(v)=0$ then 
$G$ has an $f$-parity perfect forest not containing $e$.} 
 
\2

{\em Proof of Claim~B:} Assume without loss of generality that $f(u)=0$. 
As $G$ is $2$-connected $G-u$ is connected and $\sum_{z \in V(G-u)} f(z)$ is even.  
Therefore, by Theorem~\ref{th:fpar}, there exists an 
 $f$-parity perfect forest in $G-u$, which is also an $f$-parity perfect forest in $G$ not containing the edge $e$.

\2

{\bf Claim~C:} {\em If $G$ is $2$-connected and $f(u)=f(v)=1$ then 
$G$ has a $f$-parity perfect forest if and only if $\sum_{z \in V(G)} f(z) \geq 4$.}

\2

{\em Proof of Claim~C:} Let $S=\sum_{z \in V(G)} f(z).$ As $f$ is even-sum, $S$ is even. Since $f(u)=f(v)=1$, we have $S\ge 2.$
If $S=2$ and $F$ is an $f$-parity perfect forest in $G$, then 
$u$ and $v$ must be leaves of the same tree in $F$ (as they are the only vertices with an $f$-value of one).
Therefore $e \in E(F)$, as otherwise the tree containing $u$ and $v$ is not induced in $G$. So, if
$S=2$ then $G$ has no $f$-parity perfect forest $F$ in $G$ with $e \not\in E(F)$.

We may therefore assume that $S  \geq 4$ and let $w \in V(G) \setminus \{u,v\}$ have $f(w)=1$.
As $G$ is $2$-connected there exists a $(u,v)$-path, $P$, in $G$ with $w \in V(P).$ (To see it,
consider two internally disjoint paths from $w$ to $w'$ where $w'$ is a new vertex added
to $G$ such that $N(w')=\{u,v\}$.) We now create a spanning tree $T$ in $G$, such that $E(P) \subseteq E(T)$ and
$d_T(w)=2$, as follows.  Initially let $T = P$. While $V(T) \not= V(G)$ let $q \in V(G) \setminus V(T)$ be arbitrary such 
that $q$ has an edge into $V(T) \setminus \{w\}$ (which exists as $G$ is $2$-connected). Add $q$ and an edge from $q$ into
$V(T) \setminus \{w\}$ to $T$. When $V(T)$ becomes equal to $V(G)$ we have our desired tree $T$.

Let $T_1$ and $T_2$ be the two trees in $T-w$ (there are exactly two trees in $T-w$ as $d_T(w)=2$).
Let $S_1 = \sum_{z \in V(T_1)} f(z)$ and let 
$S_2 = \sum_{z \in V(T_2)} f(z)$. As $f(w)=1$ and $V(T_1) \cup V(T_2) = V(G) \setminus \{w\}$, we note that $S_1+S_2$ is odd.
If $S_i$ is odd then add $w$ to $T_i$ ($ i \in [2]$), using the edge from $w$ to $V(T_i)$ in $T$. This results in two trees, say $T_1'$ and $T_2'$, where
$\sum_{z \in V(T_i')} f(z)$ is even for $i \in [2]$. Furthermore, as $w \in V(P)$ and $E(P) \subseteq E(T)$, we note that
$u$ and $v$ do not belong to the same tree $T_i'$. 
By Theorem~\ref{th:fpar} there exists an $f$-parity perfect forest, $F_i'$, of $G[V(T_i')]$ for $i \in [2]$
(as $T_i'$ is a spanning tree in $G[V(T_i')]$, $G[V(T_i')]$ is connected). 
Now $F_1' \cup F_2'$ is an $f$-parity perfect forest of $G$ not containing $e$.
This completes the proof of Claim~C.

\2

It is easy to see that the following algorithm is of polynomial time. 
Keep reducing the graph (see Claim~A) as long as there exists a cut-vertex and when there are
no more cut-vertices then the answer is "no" if the endpoints of $e$ have an $f$-value of one and all
other vertices have an $f$-value of zero and "yes", otherwise (see Claims~B and~C).
See Figure~\ref{IllustrationAlg} for an illustration of the algorithm.

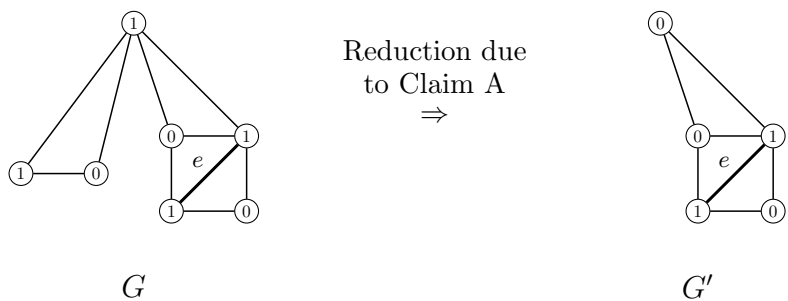
\begin{figure}[hbtp]
\input{pic7.tex}
\caption{An illustration of the algorithm given in Theorem~\ref{FparityForestAvoidingEdge}, where the 
values on the nodes indicate the $f$-values.
As in the final graph the endpoints of $e$ have an $f$-value of one and all other vertices have an $f$-value of zero 
there is no $f$-parity perfect forest in $G'$ {avoiding the edge $e$} and therefore not in $G$ either. } \label{IllustrationAlg}
\end{figure}

\section{Proof of Theorem \ref{main}}\label{sec:main}

Theorem \ref{main} follows from Theorem \ref{ThmMain} and Lemma \ref{no1PF} proved in this section. 
To prove Theorem \ref{ThmMain}, we will use the following:

\begin{lemma} \label{LemBlock}
  Let $G$ be a connected graph of even order and let $xy \in E(G)$ such that $G-\{x,y\}$ is connected.
If $G-x \in {\cal B}$ and $G-y \in {\cal B}$ then $N[x]=N[y]$.
\end{lemma}

\begin{proof}
Let $G$ be a connected graph of even order and let $xy \in E(G)$ be chosen such that $G-\{x,y\}$ is connected.
Let $G_y = G-x$ and let $G_x = G-y$ and assume that 
$G_y \in {\cal B}$ and $G_x \in {\cal B}$.
Let $C_1^x,C_2^x,\ldots , C_{l_x}^x$ be the blocks of $G_x$ and without loss
of generality assume that $x \in V(C_1^x)$. 
Analogously, let $C_1^y,C_2^y,\ldots C_{l_y}^y$ be the blocks of $G_y$ and without loss
of generality assume that $y \in V(C_1^y)$.

\2

{\bf Claim A:} {\em $N_{G_x}[x] = V(C_1^x)$ and $C_1^x$ is a complete graph of odd order and
$C_1^x - x$ is a block in $G-\{x,y\}$.
Analogously, $N_{G_x}[y] = V(C_1^y)$ and $C_1^y$ is a complete graph of odd order and
$C_1^y - y$ is a block in $G-\{x,y\}$.}

\2

{\em Proof of Claim~A:} For the sake of contradiction assume that $u_1,u_2 \in N_{G_x}(x)$ but 
$u_1$ and $u_2$ belong to different blocks of $G_x$.  In this case there is a cut-vertex in 
$G_x$ separating $u_1$ and $u_2$, which must be $x$ (as $u_1xu_2$ is a path in $G_x$).
However $x$ does not separate $u_1$ and $u_2$ as $G-\{x,y\}$ is connected. This contradiction
implies that all vertices in $N_{G_x}(x)$ belong to the same block of $G_x$.

Therefore, $N_{G_x}[x] \subseteq V(C_1^x)$ as $x$ is not a cut-vertex in $G_x$ (as $G-\{x,y\}$ is connected)
and hence $x$ only belongs to one block of $G_x$. As $G_x \in {\cal B}$ we note that
$C_1^x$ is a complete graph of odd order.
As $|V(C_1^x)| \geq 3$ (as all blocks contain at least two vertices, and $|V(C_1^x)|$ is odd) and $x$ is 
not a cut-vertex in $G_x$ we note that $C_1^x - x$ is a block in $G-\{x,y\}$.
This completes the proof of Claim~A.

\2

We now return to the proof of the lemma.  By Claim~A we note that $C_1^y-y$ is a block in $G-\{x,y\}$
which furthermore is a complete graph of even order. If $C_1^x-x$ and $C_1^y-y$  are different blocks 
in $G-\{x,y\}$, then $C_1^y-y$ is a block of even order in $G_x$, a contradiction to $G_x \in  {\cal B}$.
So, $C_1^x-x$ and $C_1^y-y$ are the same block in $G-\{x,y\}$. By Claim~A, we have
the following chain of equalities, which completes the proof of the lemma.

\[
{ N_G[x] = V(C_1^x -x) \cup \{x,y\} = V(C_1^y-y) \cup \{x,y\} = N_G[y]}
\]
\end{proof}

\begin{theorem} \label{ThmMain}
Every connected graph, $G \not\in {\cal B}$, of odd order $n\ge 3$ contains a proper $1$-perfect forest.  
\end{theorem}
\begin{proof}
The proof is by induction over odd integers $n\ge 3.$ For $n=3$, we have $G\cong P_3$, the path of order 3, which is a proper 1-perfect forest.
Now we assume that $G$ is a connected graph of odd order $n\ge 5$ such that $G \not\in {\cal B}$. Let us consider two cases.

\2

\noindent {\bf Case 1: $G$ is not 2-connected.}

\2

Assume that $G$ has a cut-vertex $x$ such that  $G-x$
has a component $C_1$ of even order. Let $G_1 = G[V(C_1) \cup \{x\}]$ and
let $G_2 = G- V(C_1)$.  Note that both $G_1$ and $G_2$ are connected graphs of odd order. 
Furthermore the set of blocks of $G$ is exactly the union of the blocks in $G_1$ and $G_2$.
As $G \not\in {\cal B}$ (and therefore some block in $G$ is not a complete graph of odd order)
 we note that either $G_1 \not\in {\cal B}$ or $G_2 \not\in {\cal B}$ (or both).

Let $i \in \{1,2\}$ be defined such that $G_i \not\in {\cal B}$ and let $j=3-i$.
By induction hypothesis, there exists a proper $1$-perfect forest $F_i$ in $G_i$.
By Theorem~\ref{Thm1perfect} there also exists a (not necessarily proper) $1$-perfect forest, $F_j$, in $G_j$,
where $x$ is the vertex of even degree in $F_j$. We now note that
$F_i \cup F_j$ is a proper $1$-perfect forest of $G$, where the only vertex of even degree is the vertex of
even degree in $F_i$. Thus, we may assume that
$G$ has no cut-vertex $x$ such that some component in $G-x$ is of even order.

Now assume that $G$ contains a cut-vertex $x$. By the previous assumption, all components in $G-x$ are of odd order, and let $C_1$ be a component of $G-x$.
Let $G_1 = G[V(C_1) \cup \{x\}]$ and let $G_2 = G- V(C_1)$.  Note that both $G_1$ and $G_2$ are connected graphs of even order.
By Theorem~\ref{ThmKnown} there exists a $0$-perfect forest $F_1$ in $G_1$
and a $0$-perfect forest $F_2$ in $G_2$.
Note that $F_1 \cup F_2$ is now a proper $1$-perfect forest of $G$, where the only vertex of even degree is $x$.

\2

\noindent {\bf Case 2: $G$ is 2-connected.}

\2

\noindent{\bf Definition A}
{\em As $G \not\in {\cal B}$ and $G$ has odd
order, we note that $G$ is not a complete graph. Therefore there exists an induced path $p_1 p_2 p_3$ in $G$ (that is, 
$p_1 p_2, p_2 p_3 \in E(G)$ and $p_1 p_3 \not\in E(G)$). Let $C_1,C_2,\ldots,C_l$ be the components in $G-\{p_2,p_3\}$,
such that $p_1 \in C_1$.}

\2

Assume first that $|V(C_1)|$ is odd.
By Theorem~\ref{Thm1perfect} there exists a $1$-perfect forest $F_1$ in $C_1$, such that $p_1$ (see Definition A)
is the vertex of even degree in $F_1$. Let $G' = G-V(C_1)$ and note that $G'$ 
is connected and of even order. Therefore, by Theorem~\ref{ThmKnown}, there exists a $0$-perfect forest, $F'$, in $G'$.

If $d_{F_1}(p_1)>0$ then $F_1 \cup F'$ is a  proper $1$-perfect forest in $G$.
Now consider the case when $d_{F_1}(p_1)=0$. As $N(p_1) \cap V(G') = \{p_2\}$ (as $p_1p_2p_3$ is an induced path in $G$)
we note that adding the edge $p_1 p_2$ to 
$F_1 \cup F'$ gives us a proper $1$-perfect forest in $G$ (where $p_2$ is the only vertex of even degree).
Thus, in the rest of the proof, we may assume that $|V(C_1)|$ is even.

Let $G' = G[V(C_1) \cup \{p_2,p_3\}]$ and note that $G$ is connected and of even order. 
Furthermore $G'-\{p_2,p_3\}$ is connected (as $G'-\{p_2,p_3\} = C_1$). 
As $p_1$ is adjacent to $p_2$ but not to $p_3$ we note that $N_{G'}[p_2] \not= N_{G'}[p_3]$.
By Lemma~\ref{LemBlock} we must therefore have $G'-p_2 \not\in {\cal B}$ or $G'-p_3 \not\in {\cal B}$.
Let $i \in \{2,3\}$ be chosen such that $G'-p_i \not\in {\cal B}$, which by induction hypothesis
implies that there is a proper $1$-perfect forest $F_1$ in $G'-p_i$. 

As $G$ is $2$-connected, we note that $p_{5-i}$ is not a cut-vertex of $G$.
Therefore every component in $G-\{p_2,p_3\}$ has an edge to $p_i$, 
which implies that $G-V(F_1)$ is connected and of even order (as both $G$ and $F_1$ are of odd order). 
By Theorem~\ref{ThmKnown} there exists a $0$-perfect forest, $F_2$, in $G-V(F_1)$. 
Now $F_1 \cup F_2$ is a proper $1$-perfect forest in $G.$
This completes the proof.
\end{proof}

\2

A semiperfect forest $F$ of $G$ is called
a {\em 2-perfect forest} if exactly two vertices of $F$ have even degree. 

\begin{lemma}\label{no1PF}
If $G$ is a connected graph of odd order and $G \in {\cal B}$ then $G$ does not contain a proper $1$-perfect forest.
\end{lemma}

\begin{proof}
Let $G$ be a connected graph of odd order and let $G \in {\cal B}$. 
We will prove that $G$ contains no proper $1$-perfect forest.
We will prove this using induction on the number of blocks in $G$.

If $G$ contains only one block then $G$ is a complete graph of odd order. In this case, any forest where all trees are induced, can
only contain trees of order 2 (and 1 if we allow isolated vertices). This implies that $G$ cannot contain a proper $1$-perfect
forest as $G$ has odd order. This completes the base case.

Now assume that $G$ contains at least two blocks, which implies that $G$ contains a cut-vertex, $x$.
Let $C_1,C_2,\ldots,C_l$ be the components in $G-x$ and let
$G_i = G[V(C_i) \cup \{x\}]$ for $i \in [l]$.
For the sake of contradiction suppose that $G$ contains a proper $1$-perfect forest $F$ 
and let $F_i$ denote $F$ restricted to $G_i$ for $i \in [l]$.
As $F$ is a proper $1$-perfect forest we note that $d_F(x) \geq 1$. 
Without loss of generality, assume that $d_{F_1}(x) \geq 1$. 
This implies that $F_1$ is a proper $i$-forest in $G_1$ where $i \in \{0,1,2\}$. 
However as $|V(G_1)|$ is odd (as $G \in {\cal B}$) this implies that 
$F_1$ is a proper $1$-perfect forest in $G_1$. 
This is a contradiction to $G_1 \in {\cal B}$ (as $G \in {\cal B}$).
\end{proof}

\2

\noindent{\bf Acknowledgements} We are thankful to Siddharth Gupta for discussions of {\sc Perfect Forest Above Perfect Matching}. 
Gutin's research was supported by the Leverhulme Trust under grant number RPG-2018-161. Yeo's research was supported by  the Danish Research Fundation under grant number DFF 7014-00037B.

\2

\end{document}

%% file: pic1.tex
\begin{center}
\tikzstyle{vertexB}=[circle,draw, top color=gray!5, bottom color=gray!30,minimum size=12pt, scale=0.6, inner sep=0.1pt]
\tikzstyle{vertexL}=[circle,draw, minimum size=15pt, scale=0.6, inner sep=0.1pt]
\begin{tikzpicture}[scale=0.45]

\node (x11) at (2,5) [vertexL] {$x_1^1$};
\node (z11) at (4,5) [vertexL] {$z_1^1$};
\node (x21) at (6,5) [vertexL] {$y_1^1$};
\node (y11) at (2,2) [vertexL] {$x_2^1$};
\node (z21) at (4,2) [vertexL] {$z_2^1$};
\node (y21) at (6,2) [vertexL] {$y_2^1$};
\node (z11p) at (4,7) [vertexB] {};
\node (x21p) at (6,7) [vertexB] {};
\node (z21p) at (4,0) [vertexB] {};
\node (y21p) at (6,0) [vertexB] {};
\draw [line width=0.02cm] (z11p) to (z11);
\draw [line width=0.02cm] (x21p) to (x21);
\draw [line width=0.02cm] (z21p) to (z21);
\draw [line width=0.02cm] (y21p) to (y21);

\node (x12) at (9,5) [vertexL] {$x_1^2$};
\node (z12) at (11,5) [vertexL] {$z_1^2$};
\node (x22) at (13,5) [vertexL] {$y_1^2$};
\node (y12) at (9,2) [vertexL] {$x_2^2$};
\node (z22) at (11,2) [vertexL] {$z_2^2$};
\node (y22) at (13,2) [vertexL] {$y_2^2$};
\node (x12p) at (9,7) [vertexB] {};
\node (z12p) at (11,7) [vertexB] {};
\node (x22p) at (13,7) [vertexB] {};
\node (y12p) at (9,0) [vertexB] {};
\node (z22p) at (11,0) [vertexB] {};
\node (y22p) at (13,0) [vertexB] {};
\draw [line width=0.02cm] (x12p) to (x12);
\draw [line width=0.02cm] (z12p) to (z12);
\draw [line width=0.02cm] (x22p) to (x22);
\draw [line width=0.02cm] (y12p) to (y12);
\draw [line width=0.02cm] (z22p) to (z22);
\draw [line width=0.02cm] (y22p) to (y22);

\node (x1n) at (20,5) [vertexL] {$x_1^n$};
\node (z1n) at (22,5) [vertexL] {$z_1^n$};
\node (x2n) at (24,5) [vertexL] {$y_1^n$};
\node (y1n) at (20,2) [vertexL] {$x_2^n$};
\node (z2n) at (22,2) [vertexL] {$z_2^n$};
\node (y2n) at (24,2) [vertexL] {$y_2^n$};
\node (x1np) at (20,7) [vertexB] {};
\node (z1np) at (22,7) [vertexB] {};
\node (y1np) at (20,0) [vertexB] {};
\node (z2np) at (22,0) [vertexB] {};
\draw [line width=0.02cm] (x1np) to (x1n);
\draw [line width=0.02cm] (z1np) to (z1n);
\draw [line width=0.02cm] (y1np) to (y1n);
\draw [line width=0.02cm] (z2np) to (z2n);

\draw [line width=0.02cm] (x11) to (z11);
\draw [line width=0.02cm] (x21) to (z11);
\draw [line width=0.02cm] (y11) to (z21);
\draw [line width=0.02cm] (y21) to (z21);

\draw [line width=0.02cm] (x11) to (y11);
\draw [line width=0.02cm] (x11) to (z21);
\draw [line width=0.02cm] (x11) to (y21);

\draw [line width=0.02cm] (z11) to (y11);
\draw [line width=0.02cm] (z11) to (z21);
\draw [line width=0.02cm] (z11) to (y21);

\draw [line width=0.02cm] (x21) to (y11);
\draw [line width=0.02cm] (x21) to (z21);
\draw [line width=0.02cm] (x21) to (y21);
\draw [line width=0.02cm] (x12) to (z12);
\draw [line width=0.02cm] (x22) to (z12);
\draw [line width=0.02cm] (y12) to (z22);
\draw [line width=0.02cm] (y22) to (z22);

\draw [line width=0.02cm] (x12) to (y12);
\draw [line width=0.02cm] (x12) to (z22);
\draw [line width=0.02cm] (x12) to (y22);

\draw [line width=0.02cm] (z12) to (y12);
\draw [line width=0.02cm] (z12) to (z22);
\draw [line width=0.02cm] (z12) to (y22);

\draw [line width=0.02cm] (x22) to (y12);
\draw [line width=0.02cm] (x22) to (z22);
\draw [line width=0.02cm] (x22) to (y22);
\draw [line width=0.02cm] (x1n) to (z1n);
\draw [line width=0.02cm] (x2n) to (z1n);
\draw [line width=0.02cm] (y1n) to (z2n);
\draw [line width=0.02cm] (y2n) to (z2n);

\draw [line width=0.02cm] (x1n) to (y1n);
\draw [line width=0.02cm] (x1n) to (z2n);
\draw [line width=0.02cm] (x1n) to (y2n);

\draw [line width=0.02cm] (z1n) to (y1n);
\draw [line width=0.02cm] (z1n) to (z2n);
\draw [line width=0.02cm] (z1n) to (y2n);

\draw [line width=0.02cm] (x2n) to (y1n);
\draw [line width=0.02cm] (x2n) to (z2n);
\draw [line width=0.02cm] (x2n) to (y2n);

\draw [line width=0.02cm] (x21) to (x12);
\draw [line width=0.02cm] (x21) to (y12);
\draw [line width=0.02cm] (y21) to (x12);
\draw [line width=0.02cm] (y21) to (y12);

\draw [dotted, line width=0.02cm] (x22) to (14,5);
\draw [dotted, line width=0.02cm] (x22) to (13.8,4.2);
\draw [dotted, line width=0.02cm] (y22) to (14,2);
\draw [dotted, line width=0.02cm] (y22) to (13.8,2.8);

\draw [dotted, line width=0.02cm] (x1n) to (19,5);
\draw [dotted, line width=0.02cm] (x1n) to (19.2,4.2);
\draw [dotted, line width=0.02cm] (y1n) to (19,2);
\draw [dotted, line width=0.02cm] (y1n) to (19.2,2.8);
\end{tikzpicture}
\end{center}

%% file: pic2.tex
\begin{center}
\tikzstyle{vertexB}=[circle,draw, top color=gray!5, bottom color=gray!30,minimum size=12pt, scale=0.6, inner sep=0.1pt]
\tikzstyle{vertexL}=[circle,draw, minimum size=15pt, scale=0.6, inner sep=0.1pt]
\begin{tikzpicture}[scale=0.45]

\node (x11) at (2,5) [vertexL] {$x_1^1$};
\node (z11) at (4,5) [vertexL] {$z_1^1$};
\node (x21) at (6,5) [vertexL] {$y_1^1$};
\node (y11) at (2,2) [vertexL] {$x_2^1$};
\node (z21) at (4,2) [vertexL] {$z_2^1$};
\node (y21) at (6,2) [vertexL] {$y_2^1$};
\node (z11p) at (4,7) [vertexB] {};
\node (x21p) at (6,7) [vertexB] {};
\node (z21p) at (4,0) [vertexB] {};
\node (y21p) at (6,0) [vertexB] {};
\draw [line width=0.02cm] (z11p) to (z11);
\draw [line width=0.02cm] (x21p) to (x21);
\draw [line width=0.02cm] (z21p) to (z21);
\draw [line width=0.02cm] (y21p) to (y21);

\node (x12) at (9,5) [vertexL] {$x_1^2$};
\node (z12) at (11,5) [vertexL] {$z_1^2$};
\node (x22) at (13,5) [vertexL] {$y_1^2$};
\node (y12) at (9,2) [vertexL] {$x_2^2$};
\node (z22) at (11,2) [vertexL] {$z_2^2$};
\node (y22) at (13,2) [vertexL] {$y_2^2$};
\node (x12p) at (9,7) [vertexB] {};
\node (z12p) at (11,7) [vertexB] {};
\node (x22p) at (13,7) [vertexB] {};
\node (y12p) at (9,0) [vertexB] {};
\node (z22p) at (11,0) [vertexB] {};
\node (y22p) at (13,0) [vertexB] {};
\draw [line width=0.02cm] (x12p) to (x12);
\draw [line width=0.02cm] (z12p) to (z12);
\draw [line width=0.02cm] (x22p) to (x22);
\draw [line width=0.02cm] (y12p) to (y12);
\draw [line width=0.02cm] (z22p) to (z22);
\draw [line width=0.02cm] (y22p) to (y22);

\node (x13) at (16,5) [vertexL] {$x_1^3$};
\node (z13) at (18,5) [vertexL] {$z_1^3$};
\node (x23) at (20,5) [vertexL] {$y_1^3$};
\node (y13) at (16,2) [vertexL] {$x_2^3$};
\node (z23) at (18,2) [vertexL] {$z_2^3$};
\node (y23) at (20,2) [vertexL] {$y_2^3$};
\node (x13p) at (16,7) [vertexB] {};
\node (z13p) at (18,7) [vertexB] {};
\node (y13p) at (16,0) [vertexB] {};
\node (z23p) at (18,0) [vertexB] {};
\draw [line width=0.02cm] (x13p) to (x13);
\draw [line width=0.02cm] (z13p) to (z13);
\draw [line width=0.02cm] (y13p) to (y13);
\draw [line width=0.02cm] (z23p) to (z23);

\draw [line width=0.02cm] (x11) to (z11);
\draw [line width=0.02cm] (x21) to (z11);
\draw [line width=0.02cm] (y11) to (z21);
\draw [line width=0.02cm] (y21) to (z21);

\draw [line width=0.02cm] (x11) to (y11);
\draw [line width=0.02cm] (x11) to (z21);
\draw [line width=0.02cm] (x11) to (y21);

\draw [line width=0.02cm] (z11) to (y11);
\draw [line width=0.02cm] (z11) to (z21);
\draw [line width=0.02cm] (z11) to (y21);

\draw [line width=0.02cm] (x21) to (y11);
\draw [line width=0.02cm] (x21) to (z21);
\draw [line width=0.02cm] (x21) to (y21);
\draw [line width=0.02cm] (x12) to (z12);
\draw [line width=0.02cm] (x22) to (z12);
\draw [line width=0.02cm] (y12) to (z22);
\draw [line width=0.02cm] (y22) to (z22);

\draw [line width=0.02cm] (x12) to (y12);
\draw [line width=0.02cm] (x12) to (z22);
\draw [line width=0.02cm] (x12) to (y22);

\draw [line width=0.02cm] (z12) to (y12);
\draw [line width=0.02cm] (z12) to (z22);
\draw [line width=0.02cm] (z12) to (y22);

\draw [line width=0.02cm] (x22) to (y12);
\draw [line width=0.02cm] (x22) to (z22);
\draw [line width=0.02cm] (x22) to (y22);
\draw [line width=0.02cm] (x13) to (z13);
\draw [line width=0.02cm] (x23) to (z13);
\draw [line width=0.02cm] (y13) to (z23);
\draw [line width=0.02cm] (y23) to (z23);

\draw [line width=0.02cm] (x13) to (y13);
\draw [line width=0.02cm] (x13) to (z23);
\draw [line width=0.02cm] (x13) to (y23);

\draw [line width=0.02cm] (z13) to (y13);
\draw [line width=0.02cm] (z13) to (z23);
\draw [line width=0.02cm] (z13) to (y23);

\draw [line width=0.02cm] (x23) to (y13);
\draw [line width=0.02cm] (x23) to (z23);
\draw [line width=0.02cm] (x23) to (y23);

\draw [line width=0.02cm] (x21) to (x12);
\draw [line width=0.02cm] (x21) to (y12);
\draw [line width=0.02cm] (y21) to (x12);
\draw [line width=0.02cm] (y21) to (y12);

\draw [line width=0.02cm] (x22) to (x13);
\draw [line width=0.02cm] (x22) to (y13);
\draw [line width=0.02cm] (y22) to (x13);
\draw [line width=0.02cm] (y22) to (y13);

\node (c1) at (24,0.5) [vertexL] {$c_1$};
\draw [line width=0.02cm] (c1) to (9,1);
\draw [line width=0.02cm] (9,1) to [out=178, in=315] (y21);
\draw [line width=0.02cm] (c1) to (y22);
\draw [line width=0.02cm] (c1) to (x23);

\node (c1p) at (23.5,-0.3) [vertexL] {$c_1'$};
\draw [line width=0.02cm] (c1p) to [out=185, in=355] (8,-0.5);
\draw [line width=0.02cm] (8,-0.5) to [out=172, in=290] (y21);
\draw [line width=0.02cm] (c1p) to (y22);
\draw [line width=0.02cm] (c1p) to (x23);

\end{tikzpicture}
\end{center}

%% file: pic3.tex
\begin{center}
\tikzstyle{vertexB}=[circle,draw, top color=gray!5, bottom color=gray!30,minimum size=12pt, scale=0.6, inner sep=0.1pt]
\tikzstyle{vertexL}=[circle,draw, minimum size=15pt, scale=0.6, inner sep=0.1pt]
\begin{tikzpicture}[scale=0.45]

\node (x11) at (2,5) [vertexL] {$x_1^1$};
\node (z11) at (4,5) [vertexL] {$z_1^1$};
\node (x21) at (6,5) [vertexL] {$y_1^1$};
\node (y11) at (2,2) [vertexL] {$x_2^1$};
\node (z21) at (4,2) [vertexL] {$z_2^1$};
\node (y21) at (6,2) [vertexL] {$y_2^1$};
\node (z11p) at (4,7) [vertexB] {};
\node (x21p) at (6,7) [vertexB] {};
\node (z21p) at (4,0) [vertexB] {};
\node (y21p) at (6,0) [vertexB] {};
\draw [line width=0.02cm] (z11p) to (z11);
\draw [line width=0.02cm] (x21p) to (x21);
\draw [line width=0.02cm] (z21p) to (z21);
\draw [line width=0.02cm] (y21p) to (y21);

\node (x12) at (9,5) [vertexL] {$x_1^2$};
\node (z12) at (11,5) [vertexL] {$z_1^2$};
\node (x22) at (13,5) [vertexL] {$y_1^2$};
\node (y12) at (9,2) [vertexL] {$x_2^2$};
\node (z22) at (11,2) [vertexL] {$z_2^2$};
\node (y22) at (13,2) [vertexL] {$y_2^2$};
\node (x12p) at (9,7) [vertexB] {};
\node (z12p) at (11,7) [vertexB] {};
\node (x22p) at (13,7) [vertexB] {};
\node (y12p) at (9,0) [vertexB] {};
\node (z22p) at (11,0) [vertexB] {};
\node (y22p) at (13,0) [vertexB] {};
\draw [line width=0.02cm] (x12p) to (x12);
\draw [line width=0.02cm] (z12p) to (z12);
\draw [line width=0.02cm] (x22p) to (x22);
\draw [line width=0.02cm] (y12p) to (y12);
\draw [line width=0.02cm] (z22p) to (z22);
\draw [line width=0.02cm] (y22p) to (y22);

\node (x13) at (16,5) [vertexL] {$x_1^3$};
\node (z13) at (18,5) [vertexL] {$z_1^3$};
\node (x23) at (20,5) [vertexL] {$y_1^3$};
\node (y13) at (16,2) [vertexL] {$x_2^3$};
\node (z23) at (18,2) [vertexL] {$z_2^3$};
\node (y23) at (20,2) [vertexL] {$y_2^3$};
\node (x13p) at (16,7) [vertexB] {};
\node (z13p) at (18,7) [vertexB] {};
\node (y13p) at (16,0) [vertexB] {};
\node (z23p) at (18,0) [vertexB] {};
\draw [line width=0.02cm] (x13p) to (x13);
\draw [line width=0.02cm] (z13p) to (z13);
\draw [line width=0.02cm] (y13p) to (y13);
\draw [line width=0.02cm] (z23p) to (z23);

\draw [line width=0.02cm] (x11) to (z11);
\draw [line width=0.02cm] (x21) to (z11);
\draw [line width=0.02cm] (y11) to (z21);
\draw [line width=0.02cm] (y21) to (z21);

\draw [line width=0.02cm] (x12) to (z12);
\draw [line width=0.02cm] (x22) to (z12);
\draw [line width=0.02cm] (y12) to (z22);
\draw [line width=0.02cm] (y22) to (z22);

\draw [line width=0.02cm] (x13) to (z13);
\draw [line width=0.02cm] (x23) to (z13);
\draw [line width=0.02cm] (y13) to (z23);
\draw [line width=0.02cm] (y23) to (z23);

\draw [line width=0.02cm] (x21) to (x12);
\draw [line width=0.02cm] (y21) to (y12);

\draw [line width=0.02cm] (x22) to (x13);
\draw [line width=0.02cm] (y22) to (y13);

\node (c1) at (24,0.5) [vertexL] {$c_1$};
\draw [line width=0.02cm] (c1) to (x23);

\node (c1p) at (23.5,-0.3) [vertexL] {$c_1'$};
\draw [line width=0.02cm] (c1p) to (x23);

\end{tikzpicture}
\end{center}

%% file: pic4.tex
\begin{center}
\tikzstyle{vertexB}=[circle,draw, top color=gray!5, bottom color=gray!30,minimum size=12pt, scale=0.6, inner sep=0.1pt]
\tikzstyle{vertexL}=[circle,draw, minimum size=15pt, scale=0.6, inner sep=0.1pt]
\begin{tikzpicture}[scale=0.35]

\node (x11) at (1,9) [vertexL] {$x_1^i$};
\node (y11) at (6,9) [vertexL] {$y_1^i$};
\node (w11) at (4,11) [vertexL] {$w_1^i$};
\node (w11o) at (3,7) [vertexL] {$\overline{w}_1^i$};
\node (x21) at (1,3) [vertexL] {$x_2^i$};
\node (y21) at (6,3) [vertexL] {$y_2^i$};
\node (w21) at (3,5) [vertexL] {$w_2^i$};
\node (w21o) at (4,1) [vertexL] {$\overline{w}_2^i$};

\draw [line width=0.02cm] (x11) to (w11);
\draw [line width=0.02cm] (x11) to (w11o);
\draw [line width=0.02cm] (y11) to (w11);
\draw [line width=0.02cm] (y11) to (w11o);
\draw [line width=0.02cm] (x21) to (w21);
\draw [line width=0.02cm] (x21) to (w21o);
\draw [line width=0.02cm] (y21) to (w21);
\draw [line width=0.02cm] (y21) to (w21o);
\draw [line width=0.02cm] (w11) to (w21o);
\draw [line width=0.02cm] (w11o) to (w21);
\end{tikzpicture} \hspace{2cm}
\begin{tikzpicture}[scale=0.5]
\node (x13) at (17,2.5) [vertexL] {$x_1^i$};
\node (y13) at (23,2.5) [vertexL] {$y_1^i$};
\node (w13) at (17,5) [vertexL] {$w_1^i$};
\node (w13o) at (21,6) [vertexL] {$\overline{w}_1^i$};
\node (x23) at (18,1) [vertexL] {$x_2^i$};
\node (y23) at (21,1) [vertexL] {$y_2^i$};
\node (w23) at (19,6) [vertexL] {$w_2^i$};
\node (w23o) at (23,5) [vertexL] {$\overline{w}_2^i$};

\draw [line width=0.02cm] (x13) to (w13);
\draw [line width=0.02cm] (x13) to (w13o);
\draw [line width=0.02cm] (y13) to (w13);
\draw [line width=0.02cm] (y13) to (w13o);
\draw [line width=0.02cm] (x23) to (w23);
\draw [line width=0.02cm] (x23) to (w23o);
\draw [line width=0.02cm] (y23) to (w23);
\draw [line width=0.02cm] (y23) to (w23o);
\draw [line width=0.02cm] (w13) to (w23o);
\draw [line width=0.02cm] (w13o) to (w23);
\end{tikzpicture}

\end{center}

%% file: pic5.tex
\begin{center}
\tikzstyle{vertexB}=[circle,draw, top color=gray!5, bottom color=gray!30,minimum size=12pt, scale=0.6, inner sep=0.1pt]
\tikzstyle{vertexL}=[circle,draw, minimum size=15pt, scale=0.6, inner sep=0.1pt]
\begin{tikzpicture}[scale=0.35]

\node (x11) at (1,2.5) [vertexL] {$x_1^1$};
\node (y11) at (7,2.5) [vertexL] {$y_1^1$};
\node (w11) at (1,5) [vertexL] {$w_1^1$};
\node (w11o) at (5,6) [vertexL] {$\overline{w}_1^1$};
\node (x21) at (2,1) [vertexL] {$x_2^1$};
\node (y21) at (5,1) [vertexL] {$y_2^1$};
\node (w21) at (3,6) [vertexL] {$w_2^1$};
\node (w21o) at (7,5) [vertexL] {$\overline{w}_2^1$};

\draw [line width=0.02cm] (x11) to (w11);
\draw [line width=0.02cm] (x11) to (w11o);
\draw [line width=0.02cm] (y11) to (w11);
\draw [line width=0.02cm] (y11) to (w11o);
\draw [line width=0.02cm] (x21) to (w21);
\draw [line width=0.02cm] (x21) to (w21o);
\draw [line width=0.02cm] (y21) to (w21);
\draw [line width=0.02cm] (y21) to (w21o);
\draw [line width=0.02cm] (w11) to (w21o);
\draw [line width=0.02cm] (w11o) to (w21);

\node (x12) at (9,2.5) [vertexL] {$x_1^2$};
\node (y12) at (15,2.5) [vertexL] {$y_1^2$};
\node (w12) at (9,5) [vertexL] {$w_1^2$};
\node (w12o) at (13,6) [vertexL] {$\overline{w}_1^2$};
\node (x22) at (10,1) [vertexL] {$x_2^2$};
\node (y22) at (13,1) [vertexL] {$y_2^2$};
\node (w22) at (11,6) [vertexL] {$w_2^2$};
\node (w22o) at (15,5) [vertexL] {$\overline{w}_2^2$};

\draw [line width=0.02cm] (x12) to (w12);
\draw [line width=0.02cm] (x12) to (w12o);
\draw [line width=0.02cm] (y12) to (w12);
\draw [line width=0.02cm] (y12) to (w12o);
\draw [line width=0.02cm] (x22) to (w22);
\draw [line width=0.02cm] (x22) to (w22o);
\draw [line width=0.02cm] (y22) to (w22);
\draw [line width=0.02cm] (y22) to (w22o);
\draw [line width=0.02cm] (w12) to (w22o);
\draw [line width=0.02cm] (w12o) to (w22);

\node (x13) at (17,2.5) [vertexL] {$x_1^3$};
\node (y13) at (23,2.5) [vertexL] {$y_1^3$};
\node (w13) at (17,5) [vertexL] {$w_1^3$};
\node (w13o) at (21,6) [vertexL] {$\overline{w}_1^3$};
\node (x23) at (18,1) [vertexL] {$x_2^3$};
\node (y23) at (21,1) [vertexL] {$y_2^3$};
\node (w23) at (19,6) [vertexL] {$w_2^3$};
\node (w23o) at (23,5) [vertexL] {$\overline{w}_2^3$};

\draw [line width=0.02cm] (x13) to (w13);
\draw [line width=0.02cm] (x13) to (w13o);
\draw [line width=0.02cm] (y13) to (w13);
\draw [line width=0.02cm] (y13) to (w13o);
\draw [line width=0.02cm] (x23) to (w23);
\draw [line width=0.02cm] (x23) to (w23o);
\draw [line width=0.02cm] (y23) to (w23);
\draw [line width=0.02cm] (y23) to (w23o);
\draw [line width=0.02cm] (w13) to (w23o);
\draw [line width=0.02cm] (w13o) to (w23);

\node (x14) at (25,2.5) [vertexL] {$x_1^4$};
\node (y14) at (31,2.5) [vertexL] {$y_1^4$};
\node (w14) at (25,5) [vertexL] {$w_1^4$};
\node (w14o) at (29,6) [vertexL] {$\overline{w}_1^4$};
\node (x24) at (26,1) [vertexL] {$x_2^4$};
\node (y24) at (29,1) [vertexL] {$y_2^4$};
\node (w24) at (27,6) [vertexL] {$w_2^4$};
\node (w24o) at (31,5) [vertexL] {$\overline{w}_2^4$};

\draw [line width=0.02cm] (x14) to (w14);
\draw [line width=0.02cm] (x14) to (w14o);
\draw [line width=0.02cm] (y14) to (w14);
\draw [line width=0.02cm] (y14) to (w14o);
\draw [line width=0.02cm] (x24) to (w24);
\draw [line width=0.02cm] (x24) to (w24o);
\draw [line width=0.02cm] (y24) to (w24);
\draw [line width=0.02cm] (y24) to (w24o);
\draw [line width=0.02cm] (w14) to (w24o);
\draw [line width=0.02cm] (w14o) to (w24);


\draw [line width=0.04cm] (x11) to (x21);
\draw (1.0,1.3) node {\footnotesize $e_1$};
\draw (30,9.7) node {\footnotesize $e_2$};

\draw [line width=0.02cm] (y11) to (x12);
\draw [line width=0.02cm] (y21) to (x22);
\draw [line width=0.02cm] (y12) to (x13);
\draw [line width=0.02cm] (y22) to (x23);
\draw [line width=0.02cm] (y13) to (x14);
\draw [line width=0.02cm] (y23) to (x24);


\node (a1) at (7,13) [vertexL] {$a^1$};
\node (c11) at (8,11) [vertexL] {$c_1^1$};
\node (c21) at (10,11) [vertexL] {$c_2^1$};
\node (c31) at (12,11) [vertexL] {$c_3^1$};
\node (b1) at (13,13) [vertexL] {$b^1$};

\draw [line width=0.02cm] (a1) to (c11);
\draw [line width=0.02cm] (a1) to (c21);
\draw [line width=0.02cm] (a1) to (c31);
\draw [line width=0.02cm] (b1) to (c11);
\draw [line width=0.02cm] (b1) to (c21);
\draw [line width=0.02cm] (b1) to (c31);

\draw [line width=0.02cm] (c11) to (w21o);
\draw [line width=0.02cm] (c11) to (w11o);
\draw [line width=0.02cm] (c21) to (w22);
\draw [line width=0.02cm] (c21) to (w12);
\draw [line width=0.02cm] (c31) to (w23);
\draw [line width=0.02cm] (c31) to (w13);


\node (a2) at (19,13) [vertexL] {$a^2$};
\node (c12) at (20,11) [vertexL] {$c_1^2$};
\node (c22) at (22,11) [vertexL] {$c_2^2$};
\node (c32) at (24,11) [vertexL] {$c_3^2$};
\node (b2) at (25,13) [vertexL] {$b^2$};

\draw [line width=0.02cm] (a2) to (c12);
\draw [line width=0.02cm] (a2) to (c22);
\draw [line width=0.02cm] (a2) to (c32);
\draw [line width=0.02cm] (b2) to (c12);
\draw [line width=0.02cm] (b2) to (c22);
\draw [line width=0.02cm] (b2) to (c32);

\draw [line width=0.02cm] (c12) to (w22o);
\draw [line width=0.02cm] (c12) to (w12o);
\draw [line width=0.02cm] (c22) to (w23o);
\draw [line width=0.02cm] (c22) to (w13o);
\draw [line width=0.02cm] (c32) to (w24);
\draw [line width=0.02cm] (c32) to (w14);

\draw [line width=0.02cm] (b1) to (a2);

\draw [line width=0.04cm] (b2)  to [out=0,in=50]   (y14);

\draw [line width=0.02cm] (y24) to [out=270,in=0]   (27,-0.2);
\draw [line width=0.02cm] (27,-0.2) to (1,-0.2);
\draw [line width=0.02cm] (1,-0.2) to [out=180,in=180] (a1);

\end{tikzpicture}
\end{center}

%% file: pic6.tex
\begin{center}
\tikzstyle{vertexB}=[circle,draw, top color=gray!5, bottom color=gray!30,minimum size=12pt, scale=0.6, inner sep=0.1pt]
\tikzstyle{vertexL}=[circle,draw, minimum size=15pt, scale=0.6, inner sep=0.1pt]
\begin{tikzpicture}[scale=0.35]

\node (x11) at (1,2.5) [vertexL] {$x_1^1$};
\node (y11) at (7,2.5) [vertexL] {$y_1^1$};
\node (w11) at (1,5) [vertexL] {$w_1^1$};
\node (w11o) at (5,6) [vertexL] {$\overline{w}_1^1$};
\node (x21) at (2,1) [vertexL] {$x_2^1$};
\node (y21) at (5,1) [vertexL] {$y_2^1$};
\node (w21) at (3,6) [vertexL] {$w_2^1$};
\node (w21o) at (7,5) [vertexL] {$\overline{w}_2^1$};

\draw [line width=0.02cm] (x11) to (w11);
\draw [line width=0.02cm] (x11) to (w11o);
\draw [line width=0.02cm] (y11) to (w11);
\draw [line width=0.02cm] (y11) to (w11o);
\draw [line width=0.02cm] (x21) to (w21);
\draw [line width=0.02cm] (x21) to (w21o);
\draw [line width=0.02cm] (y21) to (w21);
\draw [line width=0.02cm] (y21) to (w21o);
\draw [line width=0.02cm] (w11) to (w21o);
\draw [line width=0.02cm] (w11o) to (w21);

\node (x12) at (9,2.5) [vertexL] {$x_1^2$};
\node (y12) at (15,2.5) [vertexL] {$y_1^2$};
\node (w12) at (9,5) [vertexL] {$w_1^2$};
\node (w12o) at (13,6) [vertexL] {$\overline{w}_1^2$};
\node (x22) at (10,1) [vertexL] {$x_2^2$};
\node (y22) at (13,1) [vertexL] {$y_2^2$};
\node (w22) at (11,6) [vertexL] {$w_2^2$};
\node (w22o) at (15,5) [vertexL] {$\overline{w}_2^2$};

\draw [line width=0.02cm] (x12) to (w12);
\draw [line width=0.02cm] (x12) to (w12o);
\draw [line width=0.02cm] (y12) to (w12);
\draw [line width=0.02cm] (y12) to (w12o);
\draw [line width=0.02cm] (x22) to (w22);
\draw [line width=0.02cm] (x22) to (w22o);
\draw [line width=0.02cm] (y22) to (w22);
\draw [line width=0.02cm] (y22) to (w22o);
\draw [line width=0.02cm] (w12) to (w22o);
\draw [line width=0.02cm] (w12o) to (w22);

\node (x13) at (17,2.5) [vertexL] {$x_1^3$};
\node (y13) at (23,2.5) [vertexL] {$y_1^3$};
\node (w13) at (17,5) [vertexL] {$w_1^3$};
\node (w13o) at (21,6) [vertexL] {$\overline{w}_1^3$};
\node (x23) at (18,1) [vertexL] {$x_2^3$};
\node (y23) at (21,1) [vertexL] {$y_2^3$};
\node (w23) at (19,6) [vertexL] {$w_2^3$};
\node (w23o) at (23,5) [vertexL] {$\overline{w}_2^3$};

\draw [line width=0.02cm] (x13) to (w13);
\draw [line width=0.02cm] (x13) to (w13o);
\draw [line width=0.02cm] (y13) to (w13);
\draw [line width=0.02cm] (y13) to (w13o);
\draw [line width=0.02cm] (x23) to (w23);
\draw [line width=0.02cm] (x23) to (w23o);
\draw [line width=0.02cm] (y23) to (w23);
\draw [line width=0.02cm] (y23) to (w23o);
\draw [line width=0.02cm] (w13) to (w23o);
\draw [line width=0.02cm] (w13o) to (w23);

\node (x14) at (25,2.5) [vertexL] {$x_1^4$};
\node (y14) at (31,2.5) [vertexL] {$y_1^4$};
\node (w14) at (25,5) [vertexL] {$w_1^4$};
\node (w14o) at (29,6) [vertexL] {$\overline{w}_1^4$};
\node (x24) at (26,1) [vertexL] {$x_2^4$};
\node (y24) at (29,1) [vertexL] {$y_2^4$};
\node (w24) at (27,6) [vertexL] {$w_2^4$};
\node (w24o) at (31,5) [vertexL] {$\overline{w}_2^4$};

\draw [line width=0.02cm] (x14) to (w14);
\draw [line width=0.02cm] (x14) to (w14o);
\draw [line width=0.02cm] (y14) to (w14);
\draw [line width=0.02cm] (y14) to (w14o);
\draw [line width=0.02cm] (x24) to (w24);
\draw [line width=0.02cm] (x24) to (w24o);
\draw [line width=0.02cm] (y24) to (w24);
\draw [line width=0.02cm] (y24) to (w24o);
\draw [line width=0.02cm] (w14) to (w24o);
\draw [line width=0.02cm] (w14o) to (w24);


\draw [line width=0.04cm] (x11) to (x21);
\draw (1.0,1.3) node {\footnotesize $e_1$};
\draw (30,9.7) node {\footnotesize $e_2$};

\draw [line width=0.02cm] (y11) to (x12);
\draw [line width=0.02cm] (y21) to (x22);
\draw [line width=0.02cm] (y12) to (x13);
\draw [line width=0.02cm] (y22) to (x23);
\draw [line width=0.02cm] (y13) to (x14);
\draw [line width=0.02cm] (y23) to (x24);


\node (a1) at (7,13) [vertexL] {$a^1$};
\node (c11) at (8,11) [vertexL] {$c_1^1$};
\node (c21) at (10,11) [vertexL] {$c_2^1$};
\node (c31) at (12,11) [vertexL] {$c_3^1$};
\node (b1) at (13,13) [vertexL] {$b^1$};

\draw [line width=0.02cm] (a1) to (c11);
\draw [line width=0.02cm] (a1) to (c21);
\draw [line width=0.02cm] (a1) to (c31);
\draw [line width=0.02cm] (b1) to (c11);
\draw [line width=0.02cm] (b1) to (c21);
\draw [line width=0.02cm] (b1) to (c31);

\draw [line width=0.02cm] (c11) to (w21o);
\draw [line width=0.02cm] (c11) to (w11o);
\draw [line width=0.02cm] (c21) to (w22);
\draw [line width=0.02cm] (c21) to (w12);
\draw [line width=0.02cm] (c31) to (w23);
\draw [line width=0.02cm] (c31) to (w13);


\node (a2) at (19,13) [vertexL] {$a^2$};
\node (c12) at (20,11) [vertexL] {$c_1^2$};
\node (c22) at (22,11) [vertexL] {$c_2^2$};
\node (c32) at (24,11) [vertexL] {$c_3^2$};
\node (b2) at (25,13) [vertexL] {$b^2$};

\draw [line width=0.02cm] (a2) to (c12);
\draw [line width=0.02cm] (a2) to (c22);
\draw [line width=0.02cm] (a2) to (c32);
\draw [line width=0.02cm] (b2) to (c12);
\draw [line width=0.02cm] (b2) to (c22);
\draw [line width=0.02cm] (b2) to (c32);

\draw [line width=0.02cm] (c12) to (w22o);
\draw [line width=0.02cm] (c12) to (w12o);
\draw [line width=0.02cm] (c22) to (w23o);
\draw [line width=0.02cm] (c22) to (w13o);
\draw [line width=0.02cm] (c32) to (w24);
\draw [line width=0.02cm] (c32) to (w14);

\draw [line width=0.02cm] (b1) to (a2);

\draw [line width=0.04cm] (b2)  to [out=0,in=50]   (y14);

\draw [line width=0.02cm] (y24) to [out=270,in=0]   (27,-0.2);
\draw [line width=0.02cm] (27,-0.2) to (1,-0.2);
\draw [line width=0.02cm] (1,-0.2) to [out=180,in=180] (a1);


\draw [line width=0.06cm] (x11) to (w11o);
\draw [line width=0.06cm] (y11) to (w11o);
\draw [line width=0.06cm] (x21) to (w21o);
\draw [line width=0.06cm] (y21) to (w21o);

\draw [line width=0.06cm] (x12) to (w12);
\draw [line width=0.06cm] (y12) to (w12);
\draw [line width=0.06cm] (x22) to (w22);
\draw [line width=0.06cm] (y22) to (w22);


\draw [line width=0.06cm] (x13) to (w13o);
\draw [line width=0.06cm] (y13) to (w13o);
\draw [line width=0.06cm] (x23) to (w23o);
\draw [line width=0.06cm] (y23) to (w23o);

\draw [line width=0.06cm] (x14) to (w14o);
\draw [line width=0.06cm] (y14) to (w14o);
\draw [line width=0.06cm] (x24) to (w24o);
\draw [line width=0.06cm] (y24) to (w24o);


\draw [line width=0.06cm] (x11) to (x21);

\draw [line width=0.06cm] (y11) to (x12);
\draw [line width=0.06cm] (y21) to (x22);
\draw [line width=0.06cm] (y12) to (x13);
\draw [line width=0.06cm] (y22) to (x23);
\draw [line width=0.06cm] (y13) to (x14);
\draw [line width=0.06cm] (y23) to (x24);


\draw [line width=0.06cm] (a1) to (c31);
\draw [line width=0.06cm] (b1) to (c31);

\draw [line width=0.06cm] (a2) to (c12);
\draw [line width=0.06cm] (b2) to (c12);

\draw [line width=0.06cm] (b1) to (a2);
\draw [line width=0.06cm] (b2)  to [out=0,in=50]   (y14);

\draw [line width=0.06cm] (y24) to [out=270,in=0]   (27,-0.2);
\draw [line width=0.06cm] (27,-0.2) to (1,-0.2);
\draw [line width=0.06cm] (1,-0.2) to [out=180,in=180] (a1);

\end{tikzpicture}
\end{center}

%% file: pic7.tex
\begin{center}
\tikzstyle{vertexB}=[circle,draw, top color=gray!5, bottom color=gray!30,minimum size=12pt, scale=0.6, inner sep=0.1pt]
\tikzstyle{vertexL}=[circle,draw, minimum size=15pt, scale=0.6, inner sep=0.1pt]
\begin{tikzpicture}[scale=0.5]

\node (x1) at (1,2) [vertexL] {$1$};
\node (x2) at (3,2) [vertexL] {$0$};
\node (x3) at (4,6) [vertexL] {$1$};
\node (x4) at (7,3) [vertexL] {$1$};
\node (x5) at (7,1) [vertexL] {$0$};
\node (x6) at (5,1) [vertexL] {$1$};
\node (x7) at (5,3) [vertexL] {$0$};

\draw [line width=0.02cm] (x1) to (x2);
\draw [line width=0.02cm] (x1) to (x3);
\draw [line width=0.02cm] (x2) to (x3);
\draw [line width=0.02cm] (x3) to (x4);
\draw [line width=0.02cm] (x3) to (x7);

\draw [line width=0.02cm] (x4) to (x5);
\draw [line width=0.02cm] (x5) to (x6);
\draw [line width=0.02cm] (x6) to (x7);
\draw [line width=0.02cm] (x7) to (x4);
\draw [line width=0.04cm] (x4) to (x6);
\draw (5.7,2.3) node {{\footnotesize $e$}};
\draw (4,-1) node {{\large $G$}};


\draw (12,5.3) node {Reduction due};
\draw (12,4.4) node {to Claim A};
\draw (12,3.5) node {$ \Rightarrow$};


\node (x31) at (18,6) [vertexL] {$0$};
\node (x41) at (21,3) [vertexL] {$1$};
\node (x51) at (21,1) [vertexL] {$0$};
\node (x61) at (19,1) [vertexL] {$1$};
\node (x71) at (19,3) [vertexL] {$0$};

\draw [line width=0.02cm] (x31) to (x41);
\draw [line width=0.02cm] (x31) to (x71);
\draw [line width=0.02cm] (x41) to (x51);
\draw [line width=0.02cm] (x51) to (x61);
\draw [line width=0.02cm] (x61) to (x71);
\draw [line width=0.02cm] (x71) to (x41);
\draw [line width=0.04cm] (x41) to (x61);
\draw (19.7,2.3) node {{\footnotesize $e$}};
\draw (19,-1) node {{\large $G'$}};

\end{tikzpicture}
\end{center}